\documentclass[12pt]{amsart}
\usepackage{amssymb,amscd}
\usepackage{verbatim}

\usepackage{amsmath,amssymb,graphicx,mathrsfs}   % my new
\usepackage[colorlinks=true,allcolors = blue]{hyperref} % my new

\let\frak\mathfrak
\let\Bbb\mathbb

\def\>{\relax\ifmmode\mskip.666667\thinmuskip\relax\else\kern.111111em\fi}
\def\<{\relax\ifmmode\mskip-.333333\thinmuskip\relax\else\kern-.0555556em\fi}
\def\vsk#1>{\vskip#1\baselineskip}
\def\vv#1>{\vadjust{\vsk#1>}\ignorespaces}
\def\vvn#1>{\vadjust{\nobreak\vsk#1>\nobreak}\ignorespaces}

  \let\ssize\scriptstyle
\let\sssize\scriptscriptstyle

\let\Medskip\medskip
\def\medskip{\par\Medskip}
\let\Bigskip\bigskip
\def\bigskip{\par\Bigskip}

\let\Maketitle\maketitle
\def\maketitle{\Maketitle\thispagestyle{empty}\let\maketitle\empty}

\newtheorem{thm}{Theorem}[section]
\newtheorem{cor}[thm]{Corollary}
\newtheorem{lem}[thm]{Lemma}
\newtheorem{prop}[thm]{Proposition}

\theoremstyle{definition}                                  % My June 14 2017
\newtheorem{exmp}{Example}[section]

\numberwithin{equation}{section}

\theoremstyle{definition}
\newtheorem*{rem}{Remark}

\let\mc\mathcal
\let\nc\newcommand

\let\al\alpha

\let\la\lambda

\let\phi\varphi

\let\Si\Sigma

\let\der\partial

\let\geq\geqslant

\let\leq\leqslant

\let\on\operatorname
\let\bi\bibitem
\let\bs\boldsymbol

\def\C{{\mathbb C}}
\def\Z{{\mathbb Z}}
\def\R{{\mathbb R}}

\def\F{{\mathbb F}}

\def\+#1{^{\{#1\}}}

\def\id{\on{id}}

\def\Wr{\on{Wr}}

\def\beq{\begin{equation}}
\def\eeq{\end{equation}}
\def\be{\begin{equation*}}
\def\ee{\end{equation*}}

\nc{\bea}{\begin{eqnarray*}}
\nc{\eea}{\end{eqnarray*}}
\nc{\bean}{\begin{eqnarray}}
\nc{\eean}{\end{eqnarray}}
\nc{\Ref}[1]{{\rm(\ref{#1})}}

\let\ga\gamma

\nc{\Il}{{\mc I_{\bs\la}}}
\nc{\bla}{{\bs\la}}
\nc{\Fla}{\F_\bla}
\nc{\tfl}{{T^*\Fla}}
\nc{\GL}{{GL_n(\C)}}
\nc{\GLC}{{GL_n(\C)\times\C^*}}

\let\sd s %% \def\sd{\dot s}

\def\ddk_#1{\kk_{#1}\<\>\frac\der{\der\<\>\kk_{#1}}}

\def\bul{\mathbin{\raise.2ex\hbox{$\sssize\bullet$}}}
\def\intt{\mathchoice
{\mathop{\raise.2ex\rlap{$\,\,\ssize\backslash$}{\intop}}\nolimits}
{\mathop{\raise.3ex\rlap{$\,\sssize\backslash$}{\intop}}\nolimits}
{\mathop{\raise.1ex\rlap{$\sssize\>\backslash$}{\intop}}\nolimits}
{\mathop{\rlap{$\sssize\<\>\backslash$}{\intop}}\nolimits}}

\let\kk q %% Q
\let\cc c

\let\Ko K

\def\GZ/{Gelfand-Zetlin}
\def\KZ/{{\slshape KZ\/}}
\def\qKZ/{{\slshape qKZ\/}}
\def\XXX/{{\slshape XXX\/}}

\nc{\A}{{\mc C}}

{  \theoremstyle{definition}

}
\newcommand{\iso}{\overset{\sim}{\longrightarrow}}

\newcommand{\lra}{\longrightarrow}

\newcommand{\dpar}{\partial}

\newcommand{\ty}{\tilde y}
\newcommand{\tty}{\hat{y}}
\newcommand{\tCY}{\tilde{\mathcal{Y}}}

\newcommand{\fgl}{\mathfrak{gl}}

\newcommand{\fW}{W}

\newcommand{\ba}{\bs{a}}
\newcommand{\bb}{\bs{b}}
\newcommand{\bc}{\bs{c}}

\newcommand{\bh}{{\bs{h}}}

\newcommand{\bk}{{\bs{k}}}
\newcommand{\bL}{\bs{L}}

\newcommand{\bt}{\bs{t}}
\newcommand{\bu}{\bs{u}}
\newcommand{\bv}{\bs{v}}
\newcommand{\bx}{\bs{x}}
\newcommand{\by}{\bs{y}}

\newcommand{\bz}{\bs{z}}

\newcommand{\Red}{\text{Red}}

\newcommand{\CB}{\mathcal{B}}

\newcommand{\CH}{\mathcal{H}}

\newcommand{\CL}{\mathcal{L}}

\newcommand{\CN}{\mathcal{N}}

\newcommand{\CP}{\mathcal{P}}

\newcommand{\CS}{\mathcal{S}}

\newcommand{\CY}{\mathcal{Y}}
\newcommand{\CZ}{\mathcal{Z}}

\newcommand{\BA}{\mathbb{A}}
\newcommand{\BC}{\mathbb{C}}

\newcommand{\BN}{\mathbb{N}}
\newcommand{\BP}{\mathbb{P}}

\newcommand{\BR}{\mathbb{R}}
\newcommand{\BZ}{\mathbb{Z}}

\let\der\partial

\nc{\Id}{\text{Id}}
\nc\Wd{{\Wr^\dagger_V}}
\nc\Wdz{{\Wr^\dagger_{V,z}}}

\begin{document}

\hrule width0pt
\vsk->

\title[Positive Populations]
{Positive Populations}

\author
[Vadim Schechtman and Alexander Varchenko]
{ Vadim Schechtman$\>^\circ$ and Alexander Varchenko$\>^\star$}

\maketitle

\begin{center}
{\it ${}^\circ$ Institut de Math\'ematiques de Toulouse\,--\,  Universit\'e Paul Sabatier\\ 118 Route de Narbonne,
31062 Toulouse, France \/}

\vsk.5>
{\it $^{\star}\<$Department of Mathematics, University
of North Carolina at Chapel Hill\\ Chapel Hill, NC 27599-3250, USA\/}

\vsk.5>
{\it $^{\star}\<$Faculty of Mathematics and Mechanics, Lomonosov Moscow State
University\\ Leninskiye Gory 1, 119991 Moscow GSP-1, Russia\/}

\end{center}

{\let\thefootnote\relax
\footnotetext{\vsk-.8>\noindent
$^\circ\<${\sl E\>-mail}:\enspace schechtman@math.ups-tlse.fr.
\\
$^\star\<${\sl E\>-mail}:\enspace anv@email.unc.edu\>,
supported in part by NSF grant DMS-1665239}}

\begin{abstract}
A positive structure on the varieties of critical points of  master functions
for KZ equations is introduced. It comes as a combination
of the ideas from  classical works by G.\,Lusztig and  a previous
work by E.\,Mukhin and the second named author. 

\end{abstract}

\bigskip

{\it Keywords}:\ Totally positive matrices,  Whitney-Lusztig charts, Wronskian differential 
\\
\phantom{aaaaaaaaaaa}  
equation, Wronski map, master functions, Bethe cells,
positive populations

\bigskip

{\it 2010 Mathematics Subject Classification}: 13F60 (14M15, 82B23)

\bigskip

{\small \tableofcontents  }

\setcounter{footnote}{0}
\renewcommand{\thefootnote}{\arabic{footnote}}

\section{Introduction: Whitney-Lusztig patterns and Bethe populations}

\subsection{Big cell and critical points}
\label{1.1}
The aim of the present note is to introduce a positive structure on varieties of 
critical points of master functions 
arising in the integral representation for solutions of   KZ equations.

\vsk.2>
Let $N = N_{r+1}\subset G = \on{SL}_{r+1}(\BC)$ denote the group of upper triangular matrices with 
$1$'s on the diagonal. It may also be considered as a big cell in the flag variety 
$\on{SL}_{r+1}(\BC)/B_-$, where $B_-$ is the subgroup of lower triangular matrices. 
Let $S_{r+1}$ denote the Weyl group of $G$, the symmetric group. 
\vsk.2>

In this note two objects  related to $N$ will be discussed: on the one hand, 
what we call here the {\it Whitney-Loewner-Lusztig data on $N$},
on the other hand, 
a construction, introduced in \cite{MV}, which we call here {\it the Wronskian evolution} 
along the varieties of critical points.

\vsk.2> In this note we consider only the case of the group $\on{SL}_{r+1}(\C)$, although
the other reductive groups can be considered similarly.

\subsection{What is done in  Introduction}

In Section \ref{1.2} we recall the classical objects:   {\it Whitney-Loewner charts},  these are 
collections of birational coordinate systems on $N$ indexed by reduced decompositions of 
the longest element $w_0\in S_{r+1}$, and {\it Lusztig transition maps} between them.

In Sections \ref{1.3} - \ref{1.5.3} the main ideas from \cite{MV} are introduced. 
Namely, it is  a {\it reproduction}  recipe, called here a {\it Wronskian evolution},  which produces 
varieties of critical points for {\it master functions} appearing in  integral representations for
solutions of KZ equations, \cite{SV}.      

In Section \ref{1.6} the content of Sections \ref{2} - \ref{8} is described.

\subsection{Whitney-Loewner  charts and Lusztig transition maps}
\label{1.2}

In the seminal papers \cite{L,BFZ}
 Lusztig and Berenstein-Fomin-Zelevinsky 
have performed a deep study of certain remarkable coordinate systems on $N$, i.e. 
morphisms of algebraic varieties
$$
\CL_\bh:\ \BC^{q} \lra N, 
$$
$q = r(r + 1)/2$, with dense image, which are birational isomorphisms. The main feature of these 
morphisms is that the restriction of them to $\BR_{>0}^q$ induces an isomorphism
$$
\CL_\bh:\ \BR_{>0}^q \iso N_{>0}, 
$$
where $N_{>0}$ is  the subspace of {\it totally positive} upper triangular matrices.

\vsk.2>
 Recall that 
a matrix $g\in N$ is called {\it totally positive} if all its minors are strictly positive, 
except for those who are identically zero on the whole group $N$, see
\cite{BFZ}\footnote{The notion of a totally positive matrix first appeared in the works of 
I.\,Schoenberg \cite{S} and Gantmacher-Krein \cite{GK}.}.

\vsk.2>
The set of such coordinate systems, which we will be calling the {\it Whitney-Loewner charts}, is in bijection with the set $\Red(w_0)$ of reduced decompositions 
\bean
\label{1.1.1}
\bh:\ w_0 = s_{i_q}\ldots s_{i_1}
\eean
of the longest element $w_0\in S_{r+1}$. 
\vsk.2>

For example, for $r = 2$ there are two such 
coordinate systems, $\CL_{121}$ and $\CL_{212}$ corresponding to the reduced 
words $s_1s_2s_1$ and $s_2s_1s_2$ respectively. 
\vsk.2>

The construction of maps $\CL_\bh$ will be recalled below, see Section \ref{4.1}.

\vsk.2>

To every $\bh\in\Red(w_0)$ and  
$\ba = (a_q, \ldots, a_1)\in \BC^q$ there corresponds a matrix 
$$
N_\bh(\ba) = \CL_\bh(\ba)\in N.
$$
A theorem of A.\,Whitney\footnote{Anne M. Whitney (1921--2008), 
a student of Isaac Schoenberg (1903--1990).}, 
as reformulated by Ch.\,Loewner, see \cite{W, Lo}, says;

\begin{thm}
\label{thm 1.1}
For every $\bh$ each matrix $A\in N_{>0}$  is of the form $N_\bh(\ba)$
for some $\ba\in \R_{>0}^q$.
\end{thm}

\vsk.2>

For any two words $\bh, \bh'$ Lusztig has defined a {\it birational} self-map of $\BA^q$, i.e. an automorphism of the field of rational functions 
$$
\F :=  \BC(\ba) = \BC(a_1, \ldots, a_q) \cong \BC(N)
$$
(here we consider $a_i$ as independent  
transcendental generators), 
\bean
\label{1.1.2}
R_{\bh, \bh'}:\ \F \iso  \F\,,
\eean
such that 
$$
\ba' = R_{\bh, \bh'}(\ba), \quad\on{if}\ \ 
N_\bh(\ba) = N_{\bh'}(\ba')\,. 
$$ 
For example
$$
N_{121}(a_1, a_2, a_3) = \left(\begin{matrix}
1 & a_1 + a_3 & a_1a_2\\ 0 & 1 & a_2\\ 0 & 0 & 1
\end{matrix}\right) = e_1(a_1)e_2(a_2)e_1(a_3)
$$
and
$$
N_{212}(a'_1, a'_2, a'_3) = \left(\begin{matrix}
1 & a'_2 & a'_2a'_3\\ 0 & 1 & a'_1 + a'_3\\ 0 & 0 & 1
\end{matrix}\right) = e_2(a'_1)e_1(a'_2)e'_2(a'_3)\,,
$$
where $e_1(a) = 1 + ae_{12},\ e_2(a) = 1 + ae_{23}$. 

It follows that 
$$
N_{121}(a_1, a_2, a_3) = N_{212}(a'_1, a'_2, a'_3)\,,
$$
provided
\bean
\label{1.1.2a}
a'_1 = \frac{a_2a_3}{a_1 + a_3}\,,\quad a'_2 = a_1 + a_3\,,\quad a'_3 = \frac{a_2a_1}{a_1 + a_3}\,. 
\eean
This is equivalent to 
\bean
\label{1.1.2b}
a_1 = \frac{a_2\rq{}a_3\rq{}}{a_1\rq{} + a_3\rq{}}\,,
\quad 
a_2 = a_1\rq{} + a_3\rq{}\,,
\quad a_3 = \frac{a_2\rq{}a_1\rq{}}{a_1\rq{} + a_3\rq{}}\,. 
\eean
The transformation \Ref{1.1.2a} is {\it involutive}, that is, its square is 
the identity.

\subsection{Bethe Ansatz equations}
\label{1.3}

On the other hand, in the work \cite{MV} it was discovered that the variety $N$ is closely 
connected with  the varieties of critical points of certain {\it master functions} $\Phi_\bk$.  

Namely, for a sequence 
$$
\bk = (k_1, \ldots, k_r)\in \BN^r\,,
$$
consider a function $\Phi_\bk(\bu)$ depending on $k := \sum_{i=1}^r k_i$ variables subdivided 
into $r$ groups: 
$$
\bu = (u_1^{(1)}, \ldots, u^{(1)}_{k_1}; \ldots ;u^{(r)}_1, \ldots, u^{(r)}_{k_r})\,.
$$
By definition, 
$$
\Phi_\bk(\bu) = \prod_{i = 1}^r \prod_{1\leq m < l \leq k_i} 
(u_m^{(i)} - u_l^{(i)})^{a_{ii}}\cdot 
\prod_{1\leq i < j\leq r}\prod_{m=1}^{k_i}\prod_{l=1}^{k_j} (u_m^{(i)} - u_l^{(j)})^{a_{ij}}  \,.
$$
Here $A = (a_{ij})$ is the Cartan matrix for the root system of type $A_r$, in other words, 
$$
\Phi_\bk(\bu) = \prod_{i = 1}^r \prod_{1\leq l < m \leq k_i} 
(u_l^{(i)} - u_m^{(i)})^{2}\cdot 
\prod_{i=1}^{r-1}\prod_{l=1}^{k_i}
\prod_{m=1}^{k_{i+1}} (u_l^{(i)} - u_m^{(i+1)})^{-1}  \,.  
$$

Functions of this kind first appeared in \cite{SV} in  the study of integral representations for solutions  of  KZ differential equations.

A point 
$$
\bt = (t_1^{(1)}, \ldots, t^{(1)}_{k_1}; \ldots ;t^{(r)}, \ldots, t^{(r)}_{k_r})
$$
is {\it critical} for the function $\log\Phi_\bk(\bu)$ if it satisfies the system of $k$ equations 
$$
\frac{\dpar\log\Phi_\bk}{\dpar u^{(i)}_m}(\bt) = 
\biggl[\frac{\dpar \Phi_\bk}{\dpar u^{(i)}_m}\Phi_\bk^{-1}\biggr](\bt) = 0\,,
\qquad
1\leq i \leq r,\ 1\leq m \leq k_i\,,
$$
or, equivalently, 
\bean
\label{1.1.2B}
\sum_{l\neq m} \frac{a_{ii}}{t^{(i)}_m - t^{(i)}_l} + 
\sum_{j\neq i}\sum_{l = 1}^{k_j} \frac{a_{ij}}{t^{(i)}_m - t^{(j)}_l} = 0\,,
\qquad
1\leq i \leq r,\ 1\leq m \leq k_i\,.
\eean
 This system of critical point equations is also called the system of 
{\it Bethe Ansatz equations} in the Gaudin model. 

\subsection{Reproduction, or Wronskian evolution (bootstrap)} 
\label{1.4}

The following procedure of {\it reproduction} for constructing critical points 
has been proposed in \cite{ScV}, \cite{MV}. 

\vsk.2>

Let us identify the group  $\BZ^r$ with the root lattice $Q$ of the group $G$ using the base of standard simple roots $\alpha_1, \ldots, \alpha_r$.
Introduce the usual shifted action of $W = S_{r + 1}$ on $Q$:
$$
w* v = w(v - \rho) + \rho\,,
$$
where $\rho$ is the half-sum of positive roots.

Let $w\in S_{r+1}$ and let
\bean
\label{1.4.1}
\bh\ :\  w= s_{i_m}\dots s_{i_1}
\eean
be a reduced decomposition of $w$.  For any $0 \leq j \leq q$ we define an $r$-tuple 
$$
\bk^{(j)} = s_{i_j}\ldots s_{i_1}*(\mathbf{0}) \in \BN^r\,,
$$
where $\mathbf{0} = (0, \ldots, 0) = \bk^{(0)}$.  

\vsk.2>

Starting from the $r$-tuple of polynomials 
$$
\by^{(0)}=(1,\dots,1) \in \BC[x]^r\,,
$$ 
one defines inductively 
a sequence of $r$-tuples of polynomials
\bean
\label{1.4.2}
\by_\bh = (\by^{(0)}, \by^{(1)}(v_1), \by^{(2)}(v_1, v_2), \ldots, \by^{(m)}(v_1, \ldots , v_m))\,,
\eean
where 
$$
\by^{(j)}(\bv) = (y^{(j)}_1(\bv; x), \ldots, y^{(j)}_r(\bv; x))\in \BC[v_1, \ldots, v_j; x]^r,
$$
$0 \leq j \leq m$, with 
$$
\deg \by^{(j)}(\bv) := (\deg y^{(j)}_1(\bv; x), \ldots, \deg y^{(j)}_r(\bv; x)) = \bk^{(j)}\,,
$$
where $\deg$ is the degree with respect to $x$. 

\vsk.2>
The sequence $(1.4.2)$ is called the {\it population associated with a reduced word $\bh$}.  

\vsk.2>
We consider a polynomial $y^{(j)}_i(\bv; x)$ as a family $y^{(j)}_i(\bc; x)$ of polynomials  of one variable 
$x$ depending on a parameter $\bc = \bc^{(j)} = (c_1, \ldots, c_j)\in \BC^j$.  

\vsk.2>
Let $\bc\in \C^j$ and  $t_{i,1}, \ldots, t_{i, k^{(j)}_i}$ be the roots of $ y^{(j)}_i(\bc;x)$ 
ordered in any way.  Consider the tuple
$$
\bt^{(j)}(\bc) = (t_{1,1}, \ldots, t_{1, k^{(j)}_1}; \ldots ;t_{r,1}, \ldots, t_{r, k^{(j)}_r})\,.
$$

The main property of the sequence $y_\bh$ is:

\begin{thm} [\cite{MV}]
\label{1.4.1}
 For each $j=1,\dots,m$, there exists a Zarisky open dense subspace  
$U \subset \BC^j$ such that for every   
$\bc = \bc^{(j)} \in U$, the tuple of roots $\bt^{(j)}(\bc)$
is a critical point of the master function $\Phi^{(j)}(\bu) := \Phi_{\bk^{(j)}}(\bu)$, i.e. it satisfies the Bethe Ansatz equations \Ref{1.1.2B}. 

Moreover, if $\Phi_\bk$ is a master function for some index $\bk$ and $\bt$ a critical point of $\log \Phi_\bk$ 
as in Section \ref{1.3}, 
then $\bt$ appears in this construction for a reduced decomposition $\bh$ of some element $w\in S_{r+1}$.
\end{thm}

The construction of  $\by_\bh$ for $\bh \in \Red(w_0)$ see in Section \ref{1.5} below.

\subsection{Wronskian bootstrap: the details}
\label{1.5}

Starting from the $r$-tuple of polynomials 
$$
\by^{(0)}=(1,\dots,1) \in \BC[x]^r\,,
$$
one constructs the sequence 
$$
\by_\bh = (\by^{(0)}, \by^{(1)}(v_1), \by^{(2)}(v_1, v_2), \ldots, \by^{(m)}(v_1, \ldots , v_m))
\eqno{(1.5.1)}  
$$
by induction. Assume that the sequence  
$$
\by^{(j)}(\bv) = (y^{(j)}_1(\bv; x), \ldots, y^{(j)}_r(\bv; x))\in \BC[v_1, \ldots, v_j; x]^r
$$
has been constructed. 
Then the sequence
$$
\by^{(j+1)}(\bv) = (y^{(j+1)}_1(\bv; x), \ldots, y^{(j+1)}_r(\bv; x))\in \BC[v_1, \ldots, v_j,v_{j+1}; x]^r
$$
is such that 
$$
y^{(j+1)}_i(\bv; x) = y^{(j)}_i(\bv; x), \qquad \forall i\ne i_{j+1}\,.
$$
and $y^{(j+1)}_{i_{j+1}}(v_1,\dots,v_j,v_{j+1}; x)$ is constructed in two steps.

First one shows that there is a unique polynomial 
$\tilde y(v_1,\dots,v_j; x)$ such that 

\begin{enumerate}
\item[(i)]
$
\Wr(y^{(j)}_{i_{j+1}}, \tilde y)\, =\,$ const$ \,y^{(j)}_{i_{j+1}-1}y^{(j)}_{i_{j+1}+1}\,,
$
where $\Wr(f,g)=fg\rq{}-f\rq{}g$ denotes the Wronskian of two functions in $x$ and
the constant does not depend on $\bv,x$;

\item[(ii)]
the polynomial $\tilde y(v_1,\dots,v_j; x)$ is monic  with respect to the variable $x$;

\item[(iii)] the coefficient in $\tilde y(v_1,\dots,v_j; x)$ of the monomial $x^{k^{(j)}_{i_{j+1}}}$ equals zero,
where $k^{(j)}_{i_{j+1}}$ is the $i_{j+1}$-st coordinate of the vector $\bk^{(j)}$.
\end{enumerate}
Then we define
$$
y^{(j+1)}_{i_{j+1}}(v_1,\dots,v_j,v_{j+1}; x): = \tilde y(v_1,\dots,v_j; x) +
v_{j+1}\,y^{(j)}_{i_{j+1}}(v_1,\dots,v_j; x)\,.
$$

Consider all coordinates of the resulting family 
$$
\by^{(m)}(v_1, \ldots , v_m)
=(y^{(m)}_1(\bv; x), \ldots, y^{(m)}_r(\bv; x))\in \BC[v_1, \ldots, v_m; x]^r
$$
up to multiplication by nonzero numbers. This gives  a map
$$
F_\bh : \BC^m\to \BP(\BC[x])^r\,,
$$
where $\BP(\BC[x])$ is the projective space associated with $\C[x]$.
Denote by
$$
\CZ_\bh = F_\bh(\BC^m)\subset \BP(\BC[x])^r
$$
its image.

\vsk.2>
 
Let $V$ be an $r+1$-dimensional complex vector space,
$X= G/B_-$ the space of all complete flags in $V$.
Let $F_0\in X$ be a point.
The choice of $F_0$ gives rise to a decomposition of 
$X$ into $|W| = (r+1)!$ Bruhat cells, which are in bijection with $W = S_{r+1}$:
$$
X = \coprod_{w\in W} X_w\,.
$$
We have
$$
\dim X_w = \ell(w).
$$
For example, for the identity $e\in W$, the cell   $X_e = \{F_0\}$ is the  zero-dimensional cell, 
and for the longest element $w_0\in W$, the cell
 $X_{w_0}$  is the open cell, the space of all flags in general position with $F_0$.

\begin{thm} [\cite{MV}]
\label{1.5.1}
 The union 
$$
\CZ = \bigcup_{w, \bh}\  \CZ_\bh \subset \BP(\BC[x])^r
$$
 over all reduced decompositions $\bh$ of all elements $w\in S_{r+1}$ 
is an algebraic subvariety
of $ \BP(\BC[x])^r$ isomorphic to the variety of complete flags $X= G/B_-$. 

The subspace $\CZ_\bh\subset \CZ$ 
does not depend on a choice of $\bh\in\Red(w)$, 
so it may be denoted by $\CZ_w$. It is identified with the Bruhat cell $X_w\subset X$. 
\qed
\end{thm}

In particular, the subset $\CZ_{w_0} \cong X_{w_0}$ 
is isomorphic to the big Bruhat cell $N\subset G/B_-$.

\vsk.2>
The algebraic subvariety $\CZ\subset \BP(\BC[x])^r$ is what was called in \cite{MV} the {\it
population of critical points originated from} $\by^{(0)}$.

\vsk.2>
 
 The proof of Theorem \ref{1.5.1} in \cite{MV} identifies the variety $\CZ$ with the space of complete flags of a particular
$r+1$-dimensional vector space $V$, where
$$
V = V_r = \BC[x]_{\leq r} \subset \BC[x]
$$ 
is the vector space of polynomials of degree $\leq r$, and the flag $F_0$ is the standard complete flag 
$$
F_0 = (V_0\subset \ldots \subset V_{r-1}\subset V_r), \qquad  V_i = \BC[x]_{\leq i}\,.
$$

\subsection{Example, \cite[Section 3.5]{MV}}
\label{1.5.2}

For $r = 2$, let  us see how the populations  give rise to 
a decomposition of $X = \on{SL}_3(\BC)/B_-$ into six Bruhat cells. 
\vsk.2> 

We have the zero-dimensional Bruhat cell
\bea
\CZ_{\id} =\{ (1:1)\} \in \BP(\C[x])^2\,.
\eea
We have two one-dimensional  Bruhat cells:
\bea
\CZ_{s_1} &=& \{ (x+c_1\,:\,1)\ |\  c_1\in \C \} \subset \BP(\C[x])^2\,,
\\
\CZ_{s_2} &=& \{ (1\,:\,x+c_1\rq{})\ |\  c_1\rq{}\in \C \} \subset \BP(\C[x])^2\,,
\eea
two two-dimensional Bruhat cells:
\bea
\CZ_{s_2s_1} &=& \{ (x+c_1\,:\, x^2 + 2c_1x + c_2)\ |\  c_1, c_2\in \C \} \subset \BP(\C[x])^2\,,
\\
\CZ_{s_1s_2} &=& \{ (x^2 + 2c_1\rq{}x + c_2\rq{}\,:\,
x+c_1\rq{})\ |\  c_1\rq{}, c_2\rq{}\in \C \} \subset \BP(\C[x])^2\,.
\eea
The longest element $w_0\in S_3$ has two reduced decompositions $s_1s_2s_1$ and $s_2s_1s_2$, which give
two parametrizations of the same three-dimensional Bruhat cell:
\bea
\CZ_{s_1s_2s_1} &=& \{ (x^2 + c_3x + c_1c_3 - c_2    \,:\, x^2 + 2c_1x + c_2)\ |\  c_1, c_2, c_3\in \C \} \subset \BP(\C[x])^2\,,
\\
\CZ_{s_2s_1s_2} &=& \{ (x^2 + 2c_1\rq{}x + c_2\rq{}\,:\,
x^2 + c_3\rq{}x + c_1\rq{}c_3\rq{} - c_2\rq{})\ |\  c_1\rq{}, c_2\rq{}, c_3\rq{}\in \C \} \subset \BP(\C[x])^2\,.
\eea
The two coordinate systems are related by the equations
\bean
\label{1.5.2.1}
c_1\rq{}=c_3/2\,,\qquad c_2\rq{}=c_1c_3-c_2\,, \qquad c_3\rq{}=2c_1\,.
\eean
Notice that this transformation is involutive.

The union of the six Bruhat cells gives the subvariety $\CZ\subset \BP(\C[x])^2$ isomorphic to $\on{SL}_3(\BC)/B_-$.
The subvariety $\CZ$ consists of all pairs of quadratic polynomials 
$(a_2x^2+a_1x+a_0 \,:\, b_2x^2+b_1x+b_0)$ such that
\bean
\label{eqn}
a_2b_0 \,-\, \frac12a_1b_1\,+\,a_0b_2\,=\, 0\,.
\eean

\subsection{Not necessarily reduced words}
\label{1.5.3}

 One can associate to an arbitrary, not necessarily reduced 
word $\bh$ of length $m$ a sequence of $m$ $r$-tuples \Ref{1.5.1} as well, see \cite{MV}. 

Namely, using  the Wronskian differential equation
$$
W(y^{(j)}_{i_{j+1}}, \tilde y)\, = \text{const}\cdot y^{(j)}_{i_{j+1}-1}y^{(j)}_{i_{j+1}+1}
$$
as in (i) above alone, but without normalizing conditions (ii) and (iii) one gets for each $j=1,\dots,m$ a tuple
$$
\by^{(j)}(\bv) = (y^{(j)}_1(\bv; x)\,:\, \ldots\,:\, y^{(j)}_r(\bv; x)) \ \in P(\C[x])^r\,,
$$
where $\bv$ belongs to a variety of parameters $B^{(j)}$,
 which is an iterated $\BC$-torsor. This means that $B^{(j)}$ is included into a sequence of fibrations 
$$
B^{(j)} \lra B^{(j-1)} \lra \ldots \lra B^{(1)}\cong \BC\,,
$$
where each step $B^{(p)} \lra B^{(p-1)}$ is an analytic $\BC$-torsor, locally trivial in the 
usual topology. 

But the corresponding cohomology $H^1(\BC^p;\BC)$ vanishes which implies that the torsors are trivial, and this provides a global isomorphism $B^{(j)} \cong \BC^j$.

\vsk.2> 

\subsection{Main point}
The main new point of the present work is a definition of a certain {\it modified} 
reproduction, which we call the {\it normalized reproduction}. 
It provides the  varieties of $r$-tuples of polynomials $\CY^{Bethe}$, 
to be called the {\it Bethe cells}, equipped with a system 
of coordinate charts isomorphic to $N$ equipped  with  the 
Whitney-Lusztig charts.  We also define the {\it totally positive subspace}
$\mc Y^{Bethe}_{>0}\subset \mc Y^{Bethe}$ isomorphic to the subspace
$N_{>0}\subset N$ of totally positive upper triangular matrices.

\subsection{Contents of the paper}
\label{1.6}

 Section \ref{2} contains some preparations. 
In Section \ref{3}  a modification of the mutation procedure of Section \ref{1.5} is introduced.
Based on it, 
in  Section \ref{4} the {\it Bethe cell} is defined, 
and Comparison Theorem \ref{4.8} is proven, which is one of our main results. 
In Section \ref{5} we first describe in full detail Wronsky evolution for
the case of groups $\on{SL}_3$ and $\on{SL}_4$. 
In particular we compute explicitly the positive part $\CY^{Bethe}_{>0}$ of the Bethe cell, 
see Theorem \ref{5.4.1}. 
Afterwards we prove  { Triangular Theorem} \ref{5.6}, which establishes an explicit isomorphism 
between $N$ and the Bethe cell.
Section \ref{6}  contains a generalization of the previous constructions. Namely, we define the
Wronskian 
evolution, the Bethe cell, etc. associated with a finite-dimensional subspace 
$V\subset \BC[x]$ and a distinguished complete flag $F_0$ in $V$.
 The previous consideration corresponded the subspace $\BC[x]_{\leq r}$ 
of polynomials of degree $\leq r$.  
In Section \ref{7}  we present a version of the above considerations for the 
base affine space. Namely, we define a variety $\tCY^{Bethe}$, which is related to the 
previous Bethe variety $\CY^{Bethe}$ in the same way as the big cell in the base affine space $G/N_-$ is related to the 
big cell in the flag space $G/B_-$. 
In Section \ref{8}  we interpret the Whitney-Lusztig data in the language of higher Bruhat orders, 
\cite{MS}, in particular give its complete description in the crucial 
case of $\on{SL}_4$.

\medskip
 We are grateful to E.\,Mukhin and M.\,Shapiro for useful discussions.

\section{Generalities on Wronskians}
\label{2}
\subsection{Wronskian differential equation}

The Wronskian of two functions $f(x), g(x)$ is the function
$$
\Wr(f, g) = f g' - f' g = f^2(g/f)'\,.
$$
Given $f(x)$ and $h(x)$,  the equation
\bean
\label{2.1.1}
\Wr(f, g) =  h
\eean
with respect to the function $g(x)$ has a solution  
$$
g(x) =  f(x)\int h(x)f(x)^{-2}dx \,.
$$
The general solution is
\bean
\label{2.1.2}
g(x,c) =  f(x)\int_{x_0}^x h(t)f(t)^{-2}dt + c f(x),\qquad
 c\in \BC\,.
\eean

\subsection{Univaluedness}
\label{2.2}
 Let $f(x), g(x)\in \BC[x]$ be polynomials. Then
$h(x) := \Wr(f,g)$ is a polynomial.
Hence  the indefinite integral of a rational function,  
$$
\int h(x)f(x)^{-2}dx\,,
$$
has no  logarithmic terms, which is equivalent to the condition:

\medskip
\noindent
(U)\ {\it The function}  
 {\it $hf^{-2}$ has zero residues at its poles.}

\subsection{Wronskian}
\label{2.3}

Let $f_1(x), \ldots, f_n(x)$ be holomorphic functions. 
Define the Wronskian matrix 
$ (f_b^{(a-1)})_{a, b = 1}^n$
and the {\it Wronskian}
$$
\Wr(f_1, \ldots, f_n) =  \det\, (f_b^{(a-1)})_{a, b = 1}^n\,.
$$
The Wronskian is a polylinear  skew-symmetric function of $f_1, \ldots, f_n$.

\begin{exmp}
\label{2.3.1}
We have
$$
\Wr\Big(1, x, \frac{x^2}2, \ldots, \frac {x^n}{n!}\Big) \,=\, 1\,.
$$  
More generally, 
\bean
\label{5.8.1}
\Wr(x^{d_1}, \dots, x^{d_k}) = \prod_{i<j} (d_j - d_i)\,\cdot\, x^{\sum d_i - k(k+1)/2},
\eean
see this formula in \cite{MTV}.

\end{exmp}

\subsection{W5 Identity}
\label{2.4}
Let $f_1(x), f_2(x), \ldots$ be a sequence of holomorphic functions. 
For an ordered finite subset 
$$
A = \{i_1, \ldots, i_a\} \subset \BN := \{1, 2, \ldots\},
$$
we write
$$
\Wr(A) = \Wr(f_{i_1}, \ldots, f_{i_a})\,.
$$
Denote
$$
[a] = \{ 1, 2, \ldots, a\}.
$$

\begin{prop} [W5 Identity, {\cite[Section 9]{MV}}]
\label{2.4.1}

 Let 
$$
A = [a+1], \quad B = [a]\cup \{a + 2\}.
$$
Then
\bean
\Wr(\Wr(A), \Wr(B)) = \Wr(A\cap B)\cdot \Wr
(A\cup B).
\label{2.4.1}
\eean
\end{prop}

\begin{exmp}
\label{2.4.2}
We have
\bean
\Wr(\Wr(f_1,f_2), \Wr(f_1,f_3))\, =\, f_1\Wr(f_1,f_2,f_3),
\label{2.4.2}
\eean
as one can easily check.
\end{exmp}

\section{Normalized Wronskian bootstrap}
\label{3}

\subsection{Generic and fertile tuples}
Let $\by = (y_1(x), \ldots, y_r(x)) \in \BC[x]^r$
be a tuple of polynomials. Define $y_0=y_{r+1}=1$.

\vsk.2>
We say that the $r$-tuple $\by$ is {\it fertile}, if for every $i$ 
the equation
$$
\Wr(y_i(x),\hat y_i(x))=\,y_{i-1}(x)\,y_{i+1}(x)\,\,
$$
with respect to $\hat y_i(x)$ admits a polynomial solution.
We say that  $\by$ is {\it generic}, if for every $i$ the polynomial $y_i(x)$ has  no multiple roots 
and the polynomials $y_i(x)$ and $\,y_{i-1}(x)\,y_{i+1}(x)\,$ have no common roots.

\vsk.2>
Let $t_{i,1}, \ldots, t_{i, k_i}$ be the roots of $ y_i(x)$ 
ordered in any way.  Consider the tuple
$$
\bt = (t_{1,1}, \ldots, t_{1, k_1}; \ldots ;t_{r,1}, \ldots, t_{r, k_r})\,.
$$

\begin{lem}
[\cite{MV}] 
\label{lem gfer}
The tuple $\by$ is generic and fertile if and only if the tuple
$\bt$ is a critical point of the master function $\Phi_\bk$, where
$\bk=(\deg y_1,\dots,\deg y_r)$.

\end{lem}

\subsection{Normalized mutations}

 Let  $\by = (y_1, \ldots, y_r) \in \BC[x]^r$
be an $r$-tuple of polynomials such 
\bean
\label{01}
y_i(0) = 1\,,\qquad i=1,\dots,r\,. 
\eean
\vsk.3>

\begin{lem}
There exists a unique solutions $\tty_i(x)$ of the differential equation
\bean
\Wr(y_i(x), \tty_i(x)) \,=\, \,y_{i-1}(x)\,y_{i+1}(x)\,\,,
\label{3.1.1a}
\eean
such that 
\bean
\tty_i(x) = x + \mc O(x^2)\,  \quad \on{as}\quad x\to 0\,.
\label{3.1.1c}
\eean

\end{lem}

\begin{proof}
We have 
\bean
\label{form tty}
\tty_i(x)\,=\,   y_i(x) \int_{0}^x \frac{y_{i-1}(u)y_{i+1}(u)}{y_i^2(u)}\,du\,.
\eean
\end{proof}

Assume that the tuple $\by$ is generic and fertile, then the function
$\tty_i(x)$ is a polynomial by Lemma \ref{lem gfer}.
The polynomial $\tty_i(x)$ can be determined either by formula  \Ref{form tty} or
by the method of undetermined coefficients, see examples in Section \ref{5}. 

\vsk.2>
For $c\in\C$, denote
\bean
\ty_i(c;x) = y_i(x) + c\tty_i(x)\,.
\label{3.1.1b}
\eean
Notice that for any $c$ we have $\ty_i(c;0)=1$.

Define a new $r$-tuple of polynomials
\bean
\label{niy}
\nu_i(c)\by := (y_1(x), \ldots, y_{i-1}(x),\, \ty_i(c;x),\, y_{i+1}(x), \ldots,  y_r(x))
\eean
and call it {\it the $i$-th normalized mutation} of the fertile generic tuple $\by$. 
\vsk.2>

Equation \Ref{3.1.1a} with the normalizing condition \Ref{3.1.1c}
 will be called the {\it normalized Wronskian evolution}, or {\it  bootstrap, 
equation}. 

\subsection{Normalized population related to a word} 
\label{3.2}

Let
\bean
\label{3.2.11}
\bh = s_{i_m}\ldots s_{i_1}
\eean
be {\it any} word in $S_{r+1}$. 
Sometimes we will write for brevity  simply
\bean
\bh = (i_m\ldots i_1)
\label{3.2.12} 
\eean
instead of \Ref{3.2.11}. 

\vsk.2>
Now we proceed as in Section \ref{1.4},  
but will use  the normalized mutations 
$\nu_i$. Namely, we start with
$$
\by^{(0)} = (1, \ldots, 1) \ \in \C[x]^r
$$
and for each  $\bc=(c_1, \ldots, c_m)\in\C^m$ define an $r$-tuples of polynomials
 by the formula:
\bean
\by_\bh(\bc) = (y_{\bh,1}(\bc;x), \dots, y_{\bh,r}(\bc;x))\,
:=\,\nu_{i_m}(c_m)\ldots \nu_{i_1}(c_1)\by^{(0)}\,.
\label{3.2.2} 
\eean

Notice that for any $i=1,\dots,r$ and $\bc\in \C^m$ we have
\bean
\label{01c}
y_{\bh,i}(\bc;0)\,=\, 1\,.
\eean

\begin{thm}[\cite{MV}]
\label{3.3} For any $\bh$ and $\bc\in \C^m$, the tuple 
$\by_\bh(\bc) $ is a tuple of polynomials.
Moreover, there exists a Zarisky open dense subspace  
$U \subset \BC^m$ such that for every   
$\bc \in U$, the tuple of roots $\bt_\bh(\bc)$ of the polynomials
$(y_{\bh,1}(\bc;x), \dots, y_{\bh,r}(\bc;x))$
is a critical point of the corresponding  master function $\Phi_\bk$.

\end{thm}

\section{Whitney-Lusztig charts and the comparison theorem}
\label{4}

\subsection{Birational isomorphisms}
\label{4.1}

  Recall the group $N$ from Section \ref{1.1}.  

\vsk.2>
Let $e_{ij}\in \frak{gl}_{r+1}(\BC)$ denote the elementary matrix, 
$$
(e_{ij})_{ab} = \delta_{ia}\delta_{jb}\,.
$$
Define the matrices
\bean
e_i(c) = 1 + c e_{i,i+1} \in N\,,\qquad i=1,\dots,r\,,\ \ c\in \BC\,.  
\label{4.1.0}
\eean
Given a word $\bh= s_{i_m}\ldots s_{i_1}$, define a map 
\bean
\CL_\bh:\ \BC^m \lra N\,, \qquad
(c_1, \ldots c_m) \mapsto e_{i_p}(c_m)\ldots e_{i_1}(c_1)\,.
\label{4.1.1}
\eean

\vsk.2>
Suppose that the word $\bh$ is 
a reduced decomposition of the longest element $w_0$  as in \Ref{1.1.1}. 
Then $\bh$ is of length $q = r(r+1)/2$.
The corresponding map 
\bean
\CL_\bh\,:\, \BC^q \lra N\, 
\label{4.1.3}
\eean
is called the {\it Whitney-Lusztig chart} corresponding to $\bh$. 
The map $\CL_\bh$ is a birational isomorphism.

\begin{rem}
\label{4.1.1}
The map $\CL_\bh$ is {\it not} epimorphic, and  
$N$ is not even equal to the union of the images of $\CL_\bh$ for $\bh\in \Red(w_0)$. 
For example for $r = 2$ the set $\Red(w_0)$ has two elements, 
the words $(121)$ and $(212)$. It is easy to see that the matrices of the 
form $1 + ae_{13}, a\neq 0$, are {\it inaccessible}, 
$$
1 + ae_{13}\ \notin\ \CL_{121}(\BC^3) \cup \CL_{212}(\BC^3)\,.
$$

\end{rem}

\subsection{Comparison Theorem}
\label{4.2}
 Denote $M = \BC[x]^{r+1}$. 
We will write the elements of $M$ as sequences $(f_1, \ldots, f_{r+1})$ but will think of them
as column vectors. 
Then the  algebra $\fgl_{r+1}(\BC)$  of $(r+1)\times (r+1)$-matrices  acts on 
$M$ from the left in the standard way.  
\vsk.2>

We introduce a distinguished element  
\bean
\bb^{(0)} = \Big(1, x, \frac{x^2}2, \ldots, \frac {x^{r}}{r!}\Big).
\label{4.2.1}
\eean

\vsk.2>
For a word $\bh = s_{i_m}\ldots s_{i_1}$   in $S_{r+1}$
and $\bc=(c_1,\dots,c_m)\in\C^m$, define an element
\bea
\bb_\bh(\bc) = (b_{\bh,1}(\bc), \dots, b_{\bh,r+1}(\bc) )\,\in M\,,
\eea
by the formula
\bean
\label{bhc}
\bb_\bh(\bc)  := \CL_\bh(c_1, \ldots c_m)\bb^{(0)}\,.
\eean

Recall the $r$-tuple of polynomials
\bean
\by_\bh(\bc) =
(y_{\bh,1}(\bc;x),\dots,y_{\bh,r}(\bc;x)) = \nu_{i_m}(c_m)\ldots \nu_{i_1}(c_1)\by^{(0)}\,,
\label{3.2.2n} 
\eean
defined in \Ref{3.2.2}.

\begin{thm}[Comparison Theorem]
\label{thm cmp}
For any $j=1,\dots,r$, we have
\bea
\Wr(b_{\bh,1}(\bc;x), \ldots, 
b_{\bh,j}(\bc;x))\, =\, y_{\bh,j}(\bc;x)\,.
\eea

\end{thm}

\begin{proof}
The theorem is a consequence of a general statement, see Theorem \ref{4.7} below. 
\end{proof}

\subsection{Cell $\CN$}  
\label{4.4}

Denote
\bea
\CN = \mc N_r \,:=\, N\bb^{(0)} \subset M\,,
\eea
the orbit of the element $\bb^{(0)}$ under the action of $N$.

\vsk.2>

An element $\bb = (b_1, \ldots, b_{r+1})\in M$ belongs to $\CN$ if and only if 
\bean
b_i \,=\,  \frac{x^{i-1}}{(i - 1)!} \,+ \, \sum_{j=i}^r b_{ij}x^j,\qquad i=1,\dots,r+1,
\label{4.4.1}
\eean
for some $b_{ij}\in\C$.

\vsk.2>
For any $\bb = (b_1, \ldots, b_{r+1})\in \mc N$ we have
\bean
\Wr(b_1,\dots,b_{r+1}) \,=\, 1\,.
\label{4.4.2}
\eean

\vsk.2>
We define the {\it totally positive subvariety} $\mc N_{>0}\subset\mc N$
as
\bean
\label{psN}
\mc N_{>0}  = N_{>0}\bb^{(0)} \,.
\eean

\subsection{$\mc N$-$\mc Y$ correspondence}
\label{4.5}

Define the  {\it Wronski  map }
\bean
\label{WR}
{W} : \CN \lra \BC[x]^{r}\,,\quad \bb  \mapsto
(b_1, \Wr(b_1,b_2), \Wr(b_1,b_2,b_3),\dots, \Wr(b_1,\dots,b_r))\,.
\eean

\begin{lem}
\label{4.6}

If $\bb=(b_1,\dots,b_{r+1})\in \CN$
and $\by=(y_1,\dots,y_r)=\fW(\bb)$, then
\bean
\label{by i}
y_i = \Wr(b_1, \ldots, b_i) 
&=&
 1 + \mc O(x)\,, \qquad i=1,\dots,r\,,
\\
\label{by ii}
\Wr(b_1, \ldots, b_{i-1}, b_{i+1}) 
&=&
 x +\mc O(x^2)\,, \qquad i=1,\dots,r-1\,,
\eean
as $x\to 0$.
\qed
\end{lem}

We define the {\it Bethe cell} or {\it variety of Bethe $r$-tuples} as
\bean
\CY^{Bethe} := \fW(\CN)\subset \C[x]^r\,,
\label{4.6.2}
\eean
the image of the Wronski map, and  the 
{\it totally positive Bethe subvariety} or {\it positive population}
as 
$$
\CY^{Bethe}_{> 0} := \fW(\CN_{> 0})\subset \fW(\CN) =  \CY^{Bethe}\,.
$$

\begin{thm}
\label{4.7}
The Wronski map
induces an isomorphism
\bea
\fW: \ \CN \iso \CY^{Bethe}\,,
\qquad
\CN_{>0} \iso \CY^{Bethe}_{>0}\,.
\eea
\end{thm}

This theorem  is a consequence of the {Triangular Theorem} below, see Section \ref{5.6}.

\begin{thm}
\label{4.8}  
The Lusztig's mutations, i.e. mutiplications by 
$e_i(c)$ on the left are translated by $\fW$ to the Wronskian mutations $\nu_i(c)$ 
on the right:
$$
\fW(e_i(c)\bb) = \nu_i(c)\fW(\bb),\qquad i=1,\dots,r\,.
\eqno{(4.8.1)}
$$
\end{thm}

\begin{proof}
The theorem follows from the  W5 Identity.  Let us treat the case $r = 3$. 
We have 
\bea
\bb = (b_1, b_2, b_3, b_4),\quad b_i = \frac{x^{i-1}}{(i-1)!} + \ldots,
\quad
\Wr(\bb) = 1\,, \quad  \fW(\bb) = \by = (y_1, y_2, y_3, 1)\,.
\eea

\smallskip
\noindent
{\it The $i=1$ case.} We have
\bea
e_1(c)\bb 
&=& (b_1 + cb_2,  b_2, b_3, b_4),\qquad  y_1 = b_1\,,
\\
\Wr(y_1, b_2) 
&=& 
\Wr(b_1, b_2) = y_2\,.
\eea
Hence $\tty_1=b_2$ is the solution of the normalized Wronskian equations \Ref{3.1.1a}, \Ref{3.1.1c},
and
\bea
 \nu_1(c)\by = (y_1 + c\tty_1,y_2,y_3) \,=\, W(e_1(c) \by)\,,
\eea
cf. formula \Ref{niy}.

\smallskip
\noindent
{\it The $i=2$ case.}
 We have
$$
e_2(c)\bb = (b_1,  b_2 + cb_3, b_3, b_4).
$$
Denote
$\tty_2 := \Wr(b_1,  b_3)\,.$
Then 
\bea
\tty_2(x) =  x + \mc O(x^2)\,  \quad \on{as}\quad x\to 0\,,
\eea
by formula \Ref{5.8.1}, and
\bea
&&
\Wr(y_2, \tty_2) = \Wr(\Wr(b_1, b_2), \Wr(b_1, b_3)) =
 \Wr(b_1)\Wr(b_1, b_2, b_3) = y_1 y_3\,,
\eea
by the W5 Identity.
Hence $\tty_2$ is the solution of the normalized Wronskian equations \Ref{3.1.1a}, \Ref{3.1.1c},
and
\bea
 \nu_2(c)\by = (y_1, y_2 + c\tty_2,y_3) \,=\, W(e_2(c) \by)\,.
\eea

\smallskip
\noindent
{\it The $i=3$ case.}  We have
$$
e_3(c)\bb = (b_1,  b_2, b_3 + cb_4, b_4)
$$ 
Denote
$
\tty_3 := \Wr(b_1, b_2, b_4)\,.
$
Then 
\bea
\tty_3(x) =  x + \mc O(x^2)\,  \quad \on{as}\quad x\to 0\,,
\eea
by formula \Ref{5.8.1}, and
\bea
\Wr(y_3, \tty_3) = \Wr(\Wr(b_1, b_2, b_3), \Wr(b_1, b_2, b_4)) =
 \Wr(b_1, b_2)\Wr(b_1, b_2, b_3, b_4) = y_2\,,
\eea
by the W5 Identity and formula \Ref{4.4.2}.
Hence $\tty_3$ is the solution of the normalized Wronskian equations \Ref{3.1.1a}, \Ref{3.1.1c},
and
\bea
 \nu_3(c)\by = (y_1, y_2, y_3 + c\tty_3) \,=\, W(e_3(c) \by)\,.
\eea
The case of arbitrary $r$ is similar.  
\end{proof}

\section{Triangular coordinates on the Bethe cell}
\label{5}

\subsection{Example of evolutions, group $\on{SL}_3$\,} 
\label{5.1}

\subsubsection{Wronskian evolution}

We start with $\by^{(0)} = (1, 1)$ and  a reduced decomposition of the longest element in $S_3$,
$$
\bh \,= \,(121)\ :\ w_0 \,=\, s_1s_2s_1\,.
$$
The pair $\by^{(0)}$ evolves by means of the normalized  mutations: 
\bean
\label{5.1.1}
(1, 1)
& \overset{\nu_1(b_1)}
\longrightarrow
&
 (1 + b_1x, 1) \overset{\nu_2(b_2)}\longrightarrow (1 + b_1x, 1 + 
b_2(x +b_1x^2/2))
\\
\notag
&
\overset{\nu_1(b_3)}\longrightarrow 
&
(1 + b_1x + b_3(x + b_2x^2/2), 1 + 
b_2(x +b_1x^2/2))\,. 
\eean
The last pair is\ \
$ \nu_1(b_3)\nu_2(b_2)\nu_1(b_1)\,\by^{(0)}$\,.
\vsk.2>

The second reduced word, $\bh' = (212)$, gives rise to another evolution:
\bean
\label{5.1.1n2}
(1, 1) 
&\overset{\nu_2(c_1)}\longrightarrow& (1, 1 + c_1x) \overset{\nu_1(c_2)}\longrightarrow 
(1 + c_2(x +c_1x^2/2), 1 + c_1x)  
\\
&\overset{\nu_2(c_3)}\longrightarrow&
 (1 + c_2(x +c_1x^2/2), 1 + c_1x + c_3(x + c_2x^2/2)) \,. 
\notag
\eean
 
The change of coordinates on $\CY_3^{Bethe}$ from  $(b_1, b_2, b_3)$ to  $(c_1, c_2, c_3)$ is 
$$
c_2 = b_1 + b_3,\quad  c_1c_2 = b_2b_3, \qquad c_1 + c_3 = b_2, \quad c_2c_3 = b_1b_2\,,
$$
whence
\bean
c_1 = \frac{b_2b_3}{b_1 + b_3}\,,\qquad  c_2 = b_1 + b_3,\qquad c_3 = \frac{b_1b_2}{b_1 + b_3}\,,
\label{5.1.3}
\eean
cf.  \cite[Proposition 2.5]{L}.

\subsubsection{Bethe cell and positive population}

Analyzing formulas \Ref{5.1.1} and \Ref{5.1.1n2}
we observe that the Bethe cell $\mc Y^{Bethe}_3$ consists of the polynomials
$(1+e_1x+e_2/x^2/2,\,1+f_1x+f_2x^2/2)$ such that
\bean
\label{B3c}
e_2+f_2=e_1f_2\,,
\eean
cf. \Ref{eqn}, and the positive population $\mc Y^{Bethe}_{3,\,>0}\subset \mc Y^{Bethe}_3$
is cut from $\mc Y^{Bethe}_3$ by the  inequalities 
\bean
\label{posB3}
e_1,\,e_2,\,f_1,\,f_2\,>\,0\,.
\eean

\subsubsection{Whitney-Lusztig  evolution}

We start with 
\bea
\bb^{(0)}=(b_1,b_2,b_3) =(1, x, x^2/2)
\eea
and act on it  by the
 elementary unipotent matrices 
$$
e_i(t) = 1+ te_{i,i+1}, \qquad i = 1, 2\,,
$$
in the order dictated by the reduced word $\bh = (121)$: 
\bea
(1, x, x^2/2)
&\overset{e_1(t_1)}\longrightarrow&
 (1 + t_1x, x, x^2/2) 
\overset{e_2(t_2)}\longrightarrow (1 + t_1x, x + t_2x^2/2, x^2/2)
\\
&\overset{e_1(t_3)}\longrightarrow &
(1 + t_1x + t_3(x + t_2x^2/2), x + t_2x^2/2, x^2/2)\,.
\eea
The resulting triple is
\bean
\label{5.1.4}
\bb_\bh(\bt) 
&=&
(b_{\bh,1}(\bt;x), b_{\bh,2}(\bt;x), b_{\bh,3}(\bt;x))
\\
\notag
&:=& (1 + t_1x + t_3(x + t_2x^2/2), x + t_2x^2/2, x^2/2)\,,
\eean
cf. Section \ref{4.2}.
One easily checks that applying the Wronski
map $\fW$ to the evolution \Ref{5.1.4} we obtain the evolution \Ref{5.1.1}.

\subsubsection{Remark}

Note a useful  formula 
\bean
\Wr(1 + t_1x, x + t_2x^2/2) = 1 + t_2(x + t_1x^2/2)\,,
\label{5.1.5}
\eean
which could be written symbolically as 
$$
\Wr(e_1(t_1), e_2(t_2)) = e_1(t_2)e_2(t_1)\,.
$$
 In general let 
$$
\bb = (1 + ax + bx^2/2, x + cx^2/2, x^2/2)\,.
$$
Using the formulas 
$
(\Wr(1, x) , \Wr(1, x^2/2), \Wr(x, x^2/2))
=
( 1,  x,  x^2/2)\,,
$
we get
$$
\by:=\fW(\bb) = (1 + ax + bx^2/2, \ 1 + cx + (ac - b)x^2/2)\,.
$$
Thus, the  $2\times 2$-matrix of nontrivial coefficients of $\by$ has the form
$$
\left(\begin{matrix} a & b\\
c & ac - b
\end{matrix}\right).
$$
The variety $\CY^{Bethe}_3$ may be identified with
$$
\{ (a_{ij})\in \fgl_2(\BC)\ |\ a_{11}a_{21} = a_{12} + a_{22}\} \subset \fgl_2(\BC),
$$
an exchange relation familiar in the theory of cluster algebras.
Cf. Example \ref{7.4} below.

\subsection{Example of evolutions, group $\on{SL}_4$} 
\label{5.3}

\subsubsection{Wronskian evolution}
\label{5.3.2}

There are 16 distinct reduced decompositions of the longest element $w_0\in S_4$.
We choose one of them:
$$
\bh = (121321)\,.
$$	 
We start with \ $\by^{(0)} = (1, 1, 1)$
and perform 
the sequence of normalized mutations corresponding to the reduced decomposition $\bh$\,:
\bea
&&(1, 1, 1)
\overset{\nu_1(a_1)}\longrightarrow
 (1 + a_1x, 1, 1)\overset{\nu_2(a_2)}\longrightarrow 
(1 + a_1x, 1 + a_2(x + a_1x^2/2), 1)
\\
&&
\phantom{aaaa}
\overset{\nu_3(a_3)}\longrightarrow 
(1 + a_1x, 1 + a_2(x + a_1x^2/2), 
1 + a_3(x + a_2(x^2/2 + a_1x^3/6)))
\\
&&
\phantom{aa}
\overset{\nu_1(a_4)}\longrightarrow
 (1 + a_1x  + a_4(x + a_2x^2/2), 1 + a_2(x + a_1x^2/2), 
1 + a_3(x + a_2(x^2/2 + a_1x^3/6)))\,.
\eea

To find the next normalized  mutation, $\nu_2(a_5)$, we have to solve an equation
\bean
\label{b_4}
&&
\Wr(1 + a_2(x + a_1x^2/2, x + b_2x^2/2 + b_3x^3/6 + b_4x^4/24) = 
\\
\notag
&&
\phantom{aaa}
\big(1 + a_1x  + a_4(x + a_2x^2/2)\big)\cdot     
\big(1 + a_3(x + a_2(x^2 + a_1x^3/6))\big)\,
\eean
with respect to $b_2,b_3,b_4$.
We calculate inductively the coefficients of $x, x^2, x^3$ in \Ref{b_4} and obtain:
$$
b_2 = a_1 + a_3 + a_4,\quad
b_3 = 2(a_1 + a_4)a_3,\quad 
b_4 = 2a_2a_3a_4\,.
$$
Note that  the  coefficients of $x^4$ and $x^5$ in the left-hand and right-hand sides of \Ref{b_4} should be equal as well,  
but  the corresponding additional equations on $b_2, b_3, b_4$
will be satisfied identically -- these are  incarnations of  the \lq\lq{}Bethe equations\rq\rq{}.

Thus, 
\bea
&&
\nu_2(a_5)\nu_1(a_4)\nu_3(a_3)\nu_2(a_2)\nu_1(a_1)\by^{(0)}
=
(1 + a_1x  + a_4(x + a_2x^2/2), 
\\
&&
\phantom{aa}
1 + a_2(x + a_1x^2/2) + a_5(x + (a_1 + a_3 + a_4)x^2/2 + 2(a_1 + a_4)a_3x^3/6 + 
 2a_2a_3a_4 x^4/24), 
\\
&&
\phantom{aaaa}
1 + a_3(x + a_2(x^2/2 + a_1x^3/6)))\,.
\eea
Finally, to apply $\nu_1(a_6)$ to this 3-tuple,  we have to find $c_2, c_3$ from the equation
\bea
&&
\Wr(1 + a_1x  + a_4(x + a_2x^2/2), x + c_2x^2/2  + c_3x^6/6) 
\\
&&
\phantom{a}
=
1 + a_2(x + a_1x^2/2) + a_5(x + (a_1 + a_3 + a_4)x^2/2 + 2(a_1 + a_4)a_3x^3/6 + 
 2a_2a_3a_4 x^4/24\,.
\eea
We find $c_2, c_3$ by equating the coefficients of $x, x^2$:
$$
c_2 = a_2 + a_5,\qquad c_3 = a_3a_5\,.
$$
Then the  coefficients of $x^3, x^4$ will be equal as well by the \lq\lq{}Bethe equations\rq\rq{}.
Thus, 
\bea
&&
\by_\bh(\ba) = (y_{\bh,1}(\ba;x), y_{\bh,2}(\ba;x), y_{\bh,3}(\ba;x))
=\nu_1(a_6)\nu_2(a_5)\nu_1(a_4)\nu_3(a_3)\nu_2(a_2)\nu_1(a_1)\by^{(0)} 
\\
&&
\notag
= (1 + a_1x  + a_4(x + a_2x^2/2) + a_6(x + (a_2 + a_5)x^2/2 + 
a_3a_5 x^3/6), 
\\
&&
\notag
1 + a_2(x + a_1x^2/2) + a_5(x + (a_1 + a_3 + a_4)x^2/2 + 2(a_1 + a_4)a_3x^3/6 + 
 2a_2a_3a_4 x^4/24), 
\\
\notag
&&
1 + a_3(x + a_2(x^2/2 + a_1x^3/6)))\,,
\eea
cf. formula \Ref{3.2.2}.

\subsubsection{Lusztig evolution}
\label{5.3.3}
Consider the Lusztig evolution
corresponding to the same word $\bh = (121321)$: 
\bea
&&
\bb^{(0)}= (1, x, x^2/2, x^3/6) \overset{e_1(a_1)}\longrightarrow 
(1 + a_1x, x, x^2/2, x^3/6) 
\\
&&
\overset{e_2(a_2)}\longrightarrow 
(1 + a_1x,  x +  a_2x^2/2, x^2/2 , x^3/6) \overset{e_3(a_3)}\longrightarrow 
(1 + a_1x, x +  a_2x^2/2, x^2/2 + a_3x^3/6, x^3/6)  
\\
&&
\overset{e_1(a_4)}\longrightarrow
(1 + a_1x + a_4(x +  a_2x^2/2), x^2/2 + a_3x^3/6, x^3/6)
\\
&&
\overset{e_2(a_5)}\longrightarrow
(1 + a_1x + a_4(1 +  a_2x^2/2), x +  a_2x^2/2 + a_5(x^2/2 + a_3x^3/6), x^2/2 + a_3x^3/6, 
x^3/6)
\\
&&
\overset{e_1(a_6)}\longrightarrow
(1 + a_1x + a_4(1 +  a_2x^2/2) + a_6(x +  a_2x^2/2 + a_5(x^2/2 + a_3x^3/6)), 
\\
&&
\phantom{aaaaaaa}
x +  a_2x^2/2 + a_5(x^2/2 + a_3x^3/6), x^2/2 + a_3x^3/6, 
x^3/6)\,
\\
&&
=
\phantom{aa}
\bb_\bh(\ba) = (b_{\bh,1}(\ba;x), b_{\bh,2}(\ba;x), b_{\bh,3}(\ba;x), b_{\bh,4}(\ba;x))\,,
\eea
cf. formula \Ref{bhc}.

\subsubsection{Comparison}
\label{5.3.4}

We have
\bea
\by_\bh(\ba)=W(\bb_\bh(\ba))\,.
\eea
Namely,
\bea
y_{\bh,1}(\ba;x) &=& b_{\bh,1}(\ba;x)\,,
\\
y_{\bh,2}(\ba;x) &=& \Wr(b_{\bh,1}(\ba;x), b_{\bh,2}(\ba;x))\,,
\\
y_{\bh,3}(\ba;x) &=& \Wr(b_{\bh,1}(\ba;x), b_{\bh,2}(\ba;x), b_{\bh,3}(\ba;x))\,.
\eea

\subsubsection{Wronskian map and the minors of $g$}
\label{5.44}
 Let 
\bean
\label{bb54}
\bb = (1 + a_1x + a_2x^2/2 + a_3x^3/6, x + b_2x^2/2 + b_3x^3/6, 
x^2/2 + c_3x^3/6) \,\in \mc N
\eean
and
\bean
g = \left(\begin{matrix} 1 & a_1 & a_2 & a_3\\
0 & 1 & b_2 & b_3\\
0 & 0 & 1 & c_3\\
0 & 0 & 0 & 1
\end{matrix}\right) \in N\,.
\label{5.4.9}
\eean
Then  $\bb\, =\, g \bb^{(0)}$.

 Let us compute $\by =(y_1,y_2,y_3) = \fW(\bb)$ .
Clearly, 
\bean
\label{yi3}
y_1 &=& 1 + a_1x + a_2x^2/2 + a_3x^3/6\,,
\\
\notag
y_2 &=& 1 + \left|\begin{matrix} 1 & a_2\\ 0 & b_2 \end{matrix}\right| x +  
\biggl(\left|\begin{matrix} 1 & a_3\\ 0 & b_3 \end{matrix}\right| + 
\left|\begin{matrix} a_1 & a_2\\ 1 & b_2 \end{matrix}\right|\biggr) x^2/2 
\\
\notag
&&
\phantom{a}
+ \left|\begin{matrix} a_1 & a_3\\ 1 & b_3 \end{matrix}\right| x^3/3 + 
\left|\begin{matrix} a_2 & a_3\\ b_2 & b_3 \end{matrix}\right| x^4/12\,,
\\
\notag
y_3 &=& 1 + \left|\begin{matrix} 1 & a_1 & a_3\\ 0 & 1 & b_3\\ 0 & 0 & c_3 \end{matrix}\right| x + 
\left|\begin{matrix} 1 & a_2 & a_3\\ 0 & b_2 & b_3\\ 0 & 1 & c_3 \end{matrix}\right| x^2/2 + 
\left|\begin{matrix} a_1 & a_2 & a_3\\ 1 & b_2 & b_3\\ 0 & 1 & c_3 \end{matrix}\right| x^3/6\,.
\eean
Notice that all these determinants are minors of the matrix $g$. In particular, if the matrix $g$ 
is totally positive, then all these determinants are positive.

\subsubsection{Wronskian map in matrix form}
By formula \Ref{yi3}\,:
\bea
y_1 &=& 1 + a_1x + a_2x^2/2 + a_3x^3/6\,,
\\
y_2 &=& 1 + b_2 x + (b_3 + a_1b_2 - a_2)x^2/2 + (a_1b_3 - a_3)x^3/3 + 
(a_2b_3 - a_3b_2)x^4/12\,,
\\
y_3 &=& 1 + c_3x + (b_2c_3 - b_3) x^2/2 + (a_1(b_2c_3 - b_3) - (a_2c_3 - a_3)) x^3/6\,.
\eea

These formulas show that the inverse map
$$
\fW^{-1}:\ \CY^{Bethe} \iso \CN
$$
assigns to a triple $\by=(y_1,y_2,y_3)$ with
\bea
y_1
&=&
1 + \alpha_1 x + \alpha_2 x^2/2 + \alpha_3 x^3/6\,, 
\\
y_2
&=&
1 + \beta_2 x + \beta_3 x^2/2 + \beta_4 x^3/6 + \beta_5 x^4/24\,,  
\\
y_3
&=&
1 + \gamma_3 x + \gamma_4  x^2/2 + \gamma_5 x^3/6\,,
\eea
the triple $\bb = \fW^{-1}(\by)$, as in \Ref{bb54},  with
\bean
\label{5.4.6}
&&
a_1 = \alpha_1, \qquad a_2 = \alpha_2, \qquad a_3 = \alpha_3, 
\\
\notag
&&
\phantom{aaaaaaaaaaa}
b_2 = \beta_2, \qquad  b_3 = \beta_3 - (\alpha_1\beta_2 -  \alpha_2), 
\\
\notag
&&
\phantom{aaaaaaaaaaaaaaaaaaaaaa}
c_3 = \gamma_3\,.
\eean
In a matrix form, the map $\fW$ is given by the formula
\bean
\label{5.4.7}
&&
\left(\begin{matrix} 1 & \alpha_1 & \alpha_2 & \alpha_3 & 0 & 0 \\
0 & 1 &  \beta_2 &  \beta_3 & \beta_4 & \beta_5\\
0 & 0 & 1 & \gamma_3 & \gamma_4 & \gamma_5 
\end{matrix}\right) 
\\
\notag
&&
=
\left(\begin{matrix} 1 & a_1 & a_2 & a_3 & 0 & 0 \\
0 & 1 &  b_2 &  b_3 + (a_1b_2 - a_2) & 2(a_1b_3 - a_3) & 2(a_2b_3 - a_3b_2)\\
0 & 0 & 1 & c_3 & b_2c_3 - b_3 & a_1(b_2c_3 -  b_3) - (a_2c_3 - a_3)
\end{matrix}\right),
\eean
and the inverse map $\fW^{-1}$ is given by the formula
\bean
\label{5.4.8}
\left(\begin{matrix} 1 & a_1 & a_2 & a_3\\
0 & 1 & b_2 & b_3\\
0 & 0 & 1 & c_3\\
\end{matrix}\right) = \left(\begin{matrix} 1 & \alpha_1 & \alpha_2 & \alpha_3 \\
0 & 1 &  \beta_2 &  \beta_3 - (\alpha_1\beta_2 -  \alpha_2) \\
0 & 0 & 1 & \gamma_3
\end{matrix}\right)\,.
\eean

\begin{thm}
\label{thm 3 3}
The coefficients $\al_1,\al_2,\al_3,\beta_2,\beta_3,\ga_3$ of the polynomials $y_1,y_2,y_3$ serve as global coordinates
on the Bethe cell $\mc Y^{Bethe}_4$, the other coefficients $\beta_4,\beta_5,\ga_4$ are 
polynomial functions of the global coordinates due to  formulas
\Ref{5.4.7} and \Ref{5.4.8}.
\qed
\end{thm}

\begin{thm}
\label{5.4.1}

The  positive population $\CY^{Bethe}_{4,\,> 0} \subset \mc Y^{Bethe}_4$ is cut from $\mc Y^{Bethe}_4$
by the inequalities
\bean
\on{all}\quad \alpha_i, \beta_i, \gamma_i > 0
\qquad\on{and}\qquad
\beta_3 >  \alpha_1\beta_2 - \alpha_2 > 0\,.
\label{5.4.18}
\eean
\end{thm}

\begin{proof}
The proof follows from \Ref{yi3}, \Ref{5.4.7}, \Ref{5.4.8}.
\end{proof}

\subsection{Triangular coordinates}
\label{5.5}

For an arbitrary $r$, consider an element 
$\bb = (b_1, \ldots, b_{r+1})\in \CN$\,,
where
\bea
b_i \,=\,  x^{i-1}/(i - 1)! \,+ \, \sum_{j=i}^r b_{ij}x^j,\qquad i=1,\dots,r+1\,.
\eea
Denote by 
$$
M(\bb) =(b_{ij})_{1\leq i\leq j\leq r}
$$
the triangular array of nontrivial coefficients.

Let $\by$ be the corresponding Bethe tuple,
$$
\fW(\bb) = \by = (y_1, \ldots, y_r)\in \CY^{Bethe} ,
$$
with 
$$
y_i(x) = 1 + \sum_{j\geq 1} a_{ij} x^j/j!\,.
$$
Define {\it the triangular part $\by^{\triangle}$ of $\by$},
$$
\by^{\triangle} = (y_1^{\leq r}, y_2^{\leq r-1},\ldots, y_r^{\leq 1})\,.
$$
Here for a polynomial $f(x) = \sum_{i\geq 0} a_ix^i/i! \in \BC[x]$ we use the notation 
$$
f^{\leq n}(x) =  \sum_{i = 0}^n a_ix^i/i!\,.
$$
Denote
$$
A(\by^\triangle) = (a_{ij})_{1\leq i\leq  r; 1\leq j \leq r + 1 - i}\,,
$$
the triangular array of the nontrivial
coefficients of $\by^{\triangle}$; thus we take all $r$ nontrivial coefficients of $y_1(x)$; the first $r - 1$ nontrivial coefficients of $y_2(x)$, etc.

The following statement describes the relationship between
the two arrays  $M(\bb)$ and $A(\by)$.

\begin{thm}[Triangular Theorem]
\label{5.6}

${}$

\begin{enumerate}
\item[(i)]  

For all $1 \leq i\leq j \leq r$, we have
\bean
\label{ab}
a_{i,j-i+1} = b_{ij} + \phi_{ij}((b_{kl})_{k < i})\,,
\eean
where $\phi_{ij}$ is  a polynomial with all monomials of degree at least 2.

\item[(ii)] 

 Conversely, for all $1 \leq i\leq j \leq r$, we have
\bean
\label{ba}
b_{ij} = a_{i,j-i+1} + \psi_{ij}((a_{kl})_{k < i})\,,
\eean
where  $\psi_{ij}$ is  a polynomial with all monomials of degree at least 2.

\item[(iii)]
  The map
$$
\CY^{Bethe} \lra \BC^q,\quad \by \mapsto  A(\by^\triangle)\,,
$$
is an isomorphism.

\item[(iv)] The Wronski map 
$$
\fW: \CN \lra \CY^{Bethe}
$$
is an isomorphism.
\end{enumerate}

\end{thm}

\begin{proof}
All statements are corollaries of statement (i). 
We leave a proof of (i) to the reader, cf. formulas \Ref{5.4.7} and \Ref{5.4.8}.
\end{proof}

\begin{cor}
\label{cor tr coor}
The coefficients $(a_{ij})_{1\leq i\leq  r; 1\leq j \leq r + 1 - i}$
 of the polynomials $\by=(y_1,\dots,y_r)$ serve as global coordinates
on the Bethe cell $\mc Y^{Bethe}_r$, the other coefficients of $\by$
 are 
polynomial functions of the global coordinates due to  formulas
\Ref{ab} and \Ref{ba}.

\end{cor}

\subsection{Polynomials $\phi_{ij}$}
\label{5.7}

 The next statement gives  information on polynomials $\phi_{ij}(\bb)$.
Let $\bb^{(0)}$ be given by \Ref{4.2.1}. 
An arbitrary $\bb\in \CB$ has the form $\bb=g\bb^{(0)}$ for some unique $g\in N$.

\begin{thm}
\label{5.8}

The  polynomials $\phi_{ij}(\bb)$ are linear combinations, 
with strictly positive coefficients, of some minors of the matrix $g$.
Consequently, if $g$ is totally positive, then all the polynomials 
$y_i$ of the tuple $\fW(\bb)=(y_1,\dots,y_r)$ have 
strictly positive coefficients. 
\end{thm}

\begin{proof}
The proof follows from formula \Ref{5.8.1}, cf. formula \Ref{yi3}.
\end{proof}

\section{Bethe cells from subspaces of $\BC[x]$ and from critical points}
\label{6}

In this section we describe a generalization of the previous correspondence.

\subsection{From  a vector space of polynomials to a population $\mc Z_V$, see \cite{MV}}
\label{sec poavs}

Let $V\subset \C[x]$ be an $r+1$-dimensional vector space of polynomials in $x$.
We assume that $V$ has not {\it base points}, that is for any $z\in \C$ there is 
$f(x) \in V$ such that $f(z)\ne 0$.

\vsk.2> 
For any $z\in \C$ there exists a unique $r+1$-tuple of integers
$\bla=(\la_0=0<\la_1<\dots <\la_r)$
such that
for any $i=0,\dots, r$, there exists $f(x)\in V$ with the property:
\bea
\frac{d^{\la_i}f}{dx^{\la_i}}(z)\ne 0,  \qquad
\frac{d^{j}f}{dx^{j}}(z)(z) = 0,\qquad j<\la_i\,.
\eea
The tuple $\bla$ is called the {\it tuple of exponents} of $V$ at $z$.

\vsk.2>
Having $\bla$ introduce an $r$-tuple of nonnegative integers 
$\bs \mu =(\mu_1,\dots,\mu_r)$ by the formula
\bea
\mu_1 +\dots+\mu_i + i = \la_i,   \qquad i=1, \dots,r\,.
\eea
The tuple $\mu$ is called the $\frak{sl}_{r+1}$ {\it weight} of $V$ at $z$.

\vsk.2>
A point $z$ is  {\it regular} for $V$ if 
$\bla=(0,1,\dots,r)$ and hence, $\mu=(0,\dots,0)$,
otherwise the point $z$ is {\it singular}.
Denote by $\Si_V$ the set of singular points. The set $\Si_V$ is finite.
Denote
\bea
\Si_V=\{z_1,\dots,z_n\} \subset \C\,.
\eea
Denote $\bs \mu^{(a)} = ( \mu^{(a)}_1,\dots, \mu^{(a)}_r)$ the weight vector at a
 singular point $z_a$.

\vsk.2>

Introduce an $r$-tuple of polynomials $\bs T = (T_1,\dots,T_r)$,
\bean
\label{Ti}
T_i(x) = \prod_{a=1}^n (x-z_a)^{\mu^{(a)}_i}\,.
\eean
\vsk.2>

Let $\bb=(b_1,\dots,b_{r+1})$ be a basis of $V$, then
the Wronskian $\Wr(b_1,\dots,b_{r+1})$ does not depend on the choice of the basis up to multiplication
by a nonzero constant. Moreover,
\bea
\Wr(b_1,\dots,b_{r+1})\,=\, \on{const}\, T_1^rT_2^{r-1}\dots T_{r-1}^2T_r\,.
\eea

\vsk.2>
This formula shows that the set  $\Si_V$ of singular points of $V$ is the set of zeros of the Wronskian of a basis of $V$.

\vsk.2>
For any $i=2,\dots,r$ and $b_1,\dots,b_i\in V$ introduce the {\it reduced Wronskian}
\bea
\Wd(b_1,\dots,b_i) = \Wr(b_1,\dots,b_i)\,T_1^{1-i}T_2^{2-i}\dots T^{-1}_{i-1}\,.
\eea
For any $i=1,\dots,r$ and $b_1,\dots,b_i\in V$, the reduced Wronskian is a polynomial.

\vsk.2>
Introduce the {\it  reduced Wronski map}  $W^\dagger_V$, which maps
 the variety of bases of $V$ to the space of $r$-tuples of polynomials.
If $\bb=(b_1,\dots,b_{r+1})$ is a basis of $V$, then 
\bean
\label{Wd}
W^\dagger_V : \bb \mapsto  \by = (y_1,\dots,y_r):= (b_1,\Wd(b_1,b_2), \dots, \Wd(b_1,\dots,b_r))\,,
\eean
cf. \Ref{WR}.  We set $y_0=y_{r+1}=1$.

\vsk.2>

Denote
\bean
\label{tild yi}
\tty_i =\Wd(b_1,b_2,\dots, b_{i-1}b_{i+1})\,,\qquad i=1,\dots, r\,.
\eean
Then
\bean
\label{gen pr}
\Wr(y_i, \tty_i)\,=\, \on{const}\, T_i\,y_{i-1}\,y_{i+1}\,,\qquad
 i=1,\dots,r\,.
\eean

\vsk.2>
The generalized Wronski map  $W^\dagger_V$ induces a map of  the variety
 of bases of $V$ to the direct product of projective spaces
$\Bbb P(\C[x])^r$. The bases defining the same complete flag of $V$ are mapped to the same point
of $\Bbb P(\C[x])^r$. Hence the generalized Wronski map induces a map,
\bea
W^\dagger_V : X_V\to \Bbb P(\C[x])^r\,,
\eea
 of 
the variety $X_V$ of complete flags of $V$ to $\Bbb P(\C[x])^r$.

\begin{thm}[\cite{MV}]
\label{thm 6.1}

This map is an embedding.
The image 
\bean
\label{Z_V}
\mc Z_V\subset \Bbb P(\C[x])^r
\eean
 of this map is isomorphic to the variety of complete flags of $V$.
\end{thm}

The variety $\mc Z_V$
 is called the {\it population} associated with $V$, see \cite{MV}.

\vsk.2>
See in Sections \ref{1.4}-\ref{1.5.2} the example of this construction corresponding  to the case,
 where
$V=\C[x]_{\leq r}$ is the space of all the polynomials of degree $\leq r$. 
In that case the set of singular points of $V$ is empty and $T_1=\dots=T_r=1$.

\vsk.2>
A tuple $\by=(y_1,\dots,y_r)\in \C[x]^r$ is called {\it fertile}  with respect to 
$T_1,\dots,T_r$, if for any $i=1,\dots,r$ the equation
\bean
\label{gWe}
\Wr(y_i,\hat y_i)=\,T_i\, y_{i-1}\,y_{i+1}\,\,
\eean
with respect to $\hat y_i(x)$ admits a polynomial solution.
For example, all tuples $(y_1:\dots:y_r)\in\mc Z_V$ are fertile due to
\Ref{gen pr}.

\vsk.2>

A tuple $\by=(y_1,\dots,y_r)\in \C[x]^r$ is called {\it generic} with respect to 
$T_1,\dots,T_r$,  if for any $i$ the polynomial $y_i(x)$ has  no multiple roots 
and the polynomials $y_i(x)$ and $\,T_i(x) y_{i-1}(x)\,y_{i+1}(x)\,$ have no common roots.

\begin{lem}[\cite{MV}]
\label{lem gen}
All $\by \in \mc Z_V$ are fertile. Also there 
exists a Zariski open subset $U\subset \mc Z_V$ such that any
$\by\in U$ is generic.  
\end{lem}

Let $\by=(y_1,\dots, y_r)$ be a tuple of polynomials. Denote
$\bk=(k_1,\dots,k_r):=
(\deg y_1,\dots,y_r)$. Let
\bea
y_i(x)=\prod_{j=1}^{k_i} (x-t^{(i)}_j)\,, \qquad i=1,\dots,r\,,
\eea
where $t^{(i)}_j$ are the roots of $y_i$.
Denote
$\bt=(t^{(1)}_{1},\dots,t^{(1)}_{k_1};\dots;t^{(r)}_1,\dots, t^{(r)}_{k_r})$,
the tuple of roots of $\by$ ordered in any way.

\begin{thm}[\cite{MV}]
\label{thm mv gf}
A tuple $\by$ is generic and fertile with respect to $T_1,\dots,T_r$ if and only if the tuple of roots
$\bt$ is a critical points of the master function
\bean
\label{nnPhi}
\Phi_\bk(\bu,\bz,\bs\mu) 
&=& \prod_{i=1}^r\prod_{j=1}^{k_i}\prod_{a=1}^n (u^{(i)}_{j}-z_a)^{-\mu_i^{(a)}}
\\
\notag
&& 
\times 
\prod_{i = 1}^r \prod_{1\leq l < m \leq k_i} 
(u_l^{(i)} - u_m^{(i)})^{2}\cdot 
\prod_{i=1}^{r-1}\prod_{l=1}^{k_i}
\prod_{m=1}^{k_{i+1}}  (u_l^{(i)} - u_m^{(i+1)})^{-1} \,.
\eean
\end{thm}

These master functions appear in the
integral representations of the KZ equations associated with
$\frak{gl}_{r+1}$\,, see \cite{SV}.

\vsk.2>
The following statement is converse to the statement of Theorem \ref{thm mv gf}.
\vsk.2>

Given a finite set $\bz=(z_1,\dots,z_n)$, a collection of  $\bs \mu^{(a)}=( \mu^{(a)}_1,\dots, \mu^{(a)}_r)
\in \Z_{\geq 0}^r$, $a=1,\dots,n$, a vector $\bk=(k_1,\dots,k_r)\in \Z^r_{\geq 0}$,
 we define the { master function}
$\Phi_\bk(\bu,\bz,\bs\mu)$ by formula \Ref{nnPhi}.

\begin{thm}
[\cite{MV}]
\label{thm crV}
If $\bt$ is a critical point of a master function $\Phi_\bk(\bu,\bz,\bs\mu)$  with respect to the $\bu$-variables,
 then there exists a unique
$r+1$-dimensional vector space $V\subset \C[x]$,  a basis $\bb$ of $V$,
 such that the roots of the $r$-tuple of polynomials $W^\dagger_V(\bb)$ give $\bt$.

\end{thm}

Cf. Theorem \ref{1.4.1}.

\subsection{From a critical point to a population $\mc Z_V$, \cite{MV}}

Given
$\bz=(z_1,\dots,z_n)$, a collection of  $\bs \mu^{(a)}=( \mu^{(a)}_1,\dots, \mu^{(a)}_r)
\in \Z_{\geq 0}^r$, $a=1,\dots,n$, a vector $\bk=(k_1,\dots,k_r)\in \Z^r_{\geq 0}$,
let 
$\bt=(t^{(1)}_{1},\dots,t^{(1)}_{k_1};\dots;t^{(r)}_1,\dots, t^{(r)}_{k_r})$,
be  a critical point of the master function $\Phi_\bk(\bu,\bz,\bs\mu) $.
Define
\bean
\label{TTi}
T_i(x) &=&\prod_{a=1}^n(x-z_a)^{\mu^{(a)}_i}\,,\qquad i=1,\dots,r\,,
\\
\label{yit}
y_i(x)&=&\prod_{j=1}^{k_i}(x-t^{(i)}_j)\,,\qquad i=1,\dots,r\,,
\eean
and the tuple 
\bea
\by=(y_1:\dots:y_r)\ \in \Bbb P(\C[x])^r\,.
\eea
\vsk.2>

The tuple $\by$ is generic and fertile with respect to $T_1,\dots,T_r$.
Hence for every $i$, there exists a polynomial $\tty_i$ satisfying the equation
\bea
\Wr(y_i,\hat y_i)=\,\on{const}\,T_i\, y_{i-1}\,y_{i+1}\,.
\eea
Choose one solution $\tty_i$ of this equation, denote
\bea
\tilde y_i(c;x) = y_i(x)+c \hat y_i(x)\,,\qquad c\in \C\,,
\eea
and define a curve in $\Bbb P(\C[x])^r$ by the formula
\bean
\label{GEN}
\by^{(i)}(c;x) = (y_1(x):\dots : \tilde y_i(c;x) :\dots : y_r(x)) \,\in\,\Bbb P(\C[x])^r\,.
\eean
This curve is called the {\it generation} from $\by$ in the $i$-th direction.
\vsk.2>

Thus starting from the point $\by$ in $\Bbb P(\C[x])^r$ we have constructed $r$ curves in $\Bbb P(\C[x])^r$.
Now starting from any point of the
constructed  curves we may repeat this procedure and generate new $r$ curves 
in $\Bbb P(\C[x])^r$ in any of the $r$ directions. Repeating this procedure in all possible direction in any number of steps
we obtain a subset $\mc Z\subset \Bbb P(\C[x])^r$ of all points appearing in this way.
The subset $\mc Z$ is called the {\it population} generated from the critical point $\bt$

\begin{thm}
[\cite{MV}]
\label{thm crZ} 

The set $\mc Z$ is an algebraic variety isomorphic to the variety $X$ of 
complete flags in an $r+1$-dimensional vector space. 
Moreover, starting with given $\bt$ one can also determine uniquely 
an $r+1$-dimensional vector space
$V$ and a basis $\bb$ of $V$,
 such that $\by=W^\dagger_V(\bb)$ and $\mc Z=\mc Z_V$. 

\end{thm}

\subsection{Bethe cells associated with $(V;z)$} 
\label{sec bcaw}

Let $V$ be an $r+1$-dimensional vector space as in Section \ref{sec poavs}.
Let $\Si_V=\{z_1,\dots,z_n\}\subset \C$ be  the set of singular points of $V$.  Fix a  complex number
 $z\notin \Si_V$,
a regular point of $V$.

\vsk.2>
We say that a basis $\bb=(b_1,\dots,b_{r+1})$ of $V$  is a {\it unipotent basis} of $V$ with respect to $z$, if for any
$i=1,\dots,r+1$, we have
\bean
b_i \,=\,  \frac{(x-z)^{i-1}}{(i - 1)!} \,+ \, \mc O((x-z)^i)\qquad
\on{as}\qquad x\to z\,.
\label{4.4.1z}
\eean
Denote 
by $\mc N(V;z)$ the set of all unipotent bases of $V$ at $z$.

\vsk.2>
If we consider each basis $\bb$ of $V$ as an ${r+1}$-column vector, 
then the group $N$ freely acts on $\mc N(V;z)$ from the left with one orbit.
 We call $\mc N(V;z)$ {\it the cell of bases of $V$ unipotent at $z$}. 
\vsk.2>

\vsk.2>
For any $i=2,\dots,r$ and $b_1,\dots,b_i\in V$ introduce the {\it reduced Wronskian normalized at $z$}
by the formula
\bea
\Wdz(b_1,\dots,b_i) &:=& \Wd(b_1,\dots,b_i)\,T_1^{i-1}(z)T_2^{i-2}(z)\dots T_{i-1}(z)
\\
&=& \Wr(b_1,\dots,b_i)\,\frac{T_1^{i-1}(z)T_2^{i-2}(z)
\dots T_{i-1}(z)}{T_1^{i-1}(x)T_2^{i-2}(x)\dots T_{i-1}(x)}\,.
\eea
For any $i=1,\dots,r$ and $b_1,\dots,b_i\in V$, this is a polynomial.

\vsk.2>
Introduce the {\it  reduced Wronski map}  $W^\dagger_{V,z}$, which maps
 the variety of bases of $V$ to the space of $r$-tuples of polynomials.
If $\bb=(b_1,\dots,b_{r+1})$ is a basis of $V$, then 
\bean
\label{Wdz}
W^\dagger_V : \bb \mapsto  \by = (y_1,\dots,y_r):= (b_1,\Wdz(b_1,b_2), \dots, \Wdz(b_1,\dots,b_r))\,,
\eean
cf. \Ref{WR}.  We set $y_0=y_{r+1}=1$.

\vsk.2>
If $\bb =(b_1, \dots,b_{r+1})\in \mc N(V;z)$, then
\bean
\label{yi(z)0=1 }
y_i(z)=1\,,\qquad i=1,\dots, r\,.
\eean

\vsk.2>
Introduce the {\it Bethe cell} $\mc Y^{Bethe}(V;z)$ as the image of the cell $\mc N(V;z)$ 
under the reduced Wronski map $W^\dagger_{V,z}$\,,
\bean
\label{BcVz} 
\mc Y^{Bethe}(V;z)\, :=\, W^\dagger_{V,z}(\mc N(V;z))\, \subset \,\C[x]^r\,.
\eean

\begin{thm}
\label{thm NYB}
The reduced Wronski map $W^\dagger_{V,z}$ induces an isomorphism
\bean
\label{rW iso}
W^\dagger_{V,z}\,:\,\mc N(V;z)\, \to\, \mc Y^{Bethe}(V;z)\,.
\eean

\end{thm}

\begin{proof}
The theorem is deduced from Theorem \ref{thm 6.1}, or it can be proved along the lines of the proof 
of Triangular Theorem \ref{5.6}, which 
corresponds to the subspace
$V = \BC[x]_{\leq r}$ of polynomials of degree $\leq r$ and $z = 0$.
\end{proof}

\begin{cor}
\label{cor actN}
The action of $N$ on $\mc N(V;z)$ and the Wronski map 
$W^\dagger_{V,z}$ induce an action of $N$ on the Bethe cell
$\mc Y^{Bethe}(V;z)$.

\end{cor}

\begin{lem}
\label{lem genB}
There exists a Zariski open subset $U\subset \mc Y^{Bethe}(V;z)$, such that any 
$\by\in U$ is generic and fertile.
\end{lem}

\begin{proof}
The theorem follows from Lemma \ref{lem gen}.
\end{proof}

Recall that by Theorem \ref{thm mv gf}, if $\by=(y_1,\dots,y_r)$ is generic and fertile, then roots of 
these polynomials determine a critical point of a suitable master function.

\subsection{Normalized mutations and $\mc N$-$\mc Y$ correspondence}
\label{sec Nmut}

Let $\by=(y_1,\dots,y_r) \in \mc Y^{Bethe}(V;z)$ and $i=1,\dots,r$. Define the {\it normalized mutation } of $\by$ in the
$i$-th direction. 
Consider the differential equation
\bean
\label{nm i}
\Wr(y_i(x),\hat y_i(x)) \,=\,\frac{T_i(x)}{T_i(z)}\,y_{i-1}(x)\,y_{i+1}(x), 
\eean
with respect to $\hat y_i$ with initial condition
\bean
\label{ini con}
\hat y_i(z) = 0,\qquad \frac{d\hat y_i}{d x}(z)\, = \,1 .
\eean
It has the unique solution
\bean
\label{uni sol}
\hat y_i (x) \,=\, y_i(x)\int_{z}^x\frac {T_i(u) y_{i-1}(u) y_{i+1}(u)}{T_i(z) y_i(u)^2}\,du\,.
\eean
This solution is a polynomial, cf. equation \Ref{gen pr}.
For $c\in\C$, denote
\bean
\ty_i(c;x) = y_i(x) + c\tty_i(x)\,.
\label{3.1.1bn}
\eean
Notice that  $\ty_i(c;z)=1$.

Define a new $r$-tuple of polynomials
\bean
\label{niyn}
\nu_i(c)\by := (y_1(x), \ldots, y_{i-1}(x),\, \ty_i(c;x),\, y_{i+1}(x), \ldots,  y_r(x))\,.
\eean
We call it {\it the $i$-th normalized mutation} of the  tuple $\by\in \mc N(V;z)$. 
\vsk.2>

\begin{lem}
For any $i$, $c$ the tuple $\nu_i(c)\by$ lies in $\mc Y^{Bethe}(V;z)$.
\end{lem}

\begin{proof} 
The lemma follows from formula \Ref{gen pr}.
\end{proof}

By this lemma we have a map
\bea
\nu_i(c) : \mc Y^{Bethe}(V;z)\to \mc Y^{Bethe}(V;z)\,.
\eea

Recall the unipotent matrices $e_i(c)$, $i=1,\dots,r$, $c\in\C$, introduced in \Ref{4.1.0}.

\begin{thm}[Comparison Theorem]
\label{thm compTg}
Let $\bb \in \mc N(V;z)$, $i=1,\dots,r$, $c\in \C$. Then
\bean
\label{comp f}
W^\dagger_{V,z}(e_i(c)\bb) \,=\, 
\mu_i(c)W^\dagger_{V,z}(\bb) \,.
\eean

\end{thm}

\begin{proof}
The proof follows from formula \Ref{gen pr}. Cf. the proof of Theorem \ref{4.8}.
\end{proof}

\begin{cor}
By Corollary \ref{cor actN} the group $N$ acts on the Bethe cell $\mc Y^{Bethe}(V;z)$.
By Theorem \ref{thm compTg} a unipotent matrix $e_i(c)$ acts 
on the Bethe cell $\mc Y^{Bethe}(V;z)$
as the normalized mutation
 $\mu_i(c)$.
\end{cor}

\subsection{Positive populations}
Fix
any point $\bb^{(0)}\in \mc N(V;z)$, that is any basis of $V$ unipotent at $z$.
Then 
\bea
N \bb^{(0)} = \mc N(V,z)\,.
\eea
Define the {\it totally positive part} of $\mc N(V,z)$ as
\bea
\mc N_{>0}(V;z;\bb^{(0)})\,:=\, N_{>0} \bb^{(0)}\,.
\eea
The totally positive part depends on the choice of $\bb^{(0)}$.

\vsk.2>
Denote 
\bea
\by^{(0)} \,:=\, W^\dagger_{V,z} (\bb^{(0)})\,\in\, \CY^{Bethe}(V;z)\,.
\eea
Since $\bb^{(0)}$ is any point of $\mc N(V;z)$, the point 
$\by^{(0)}$ could be an any point of the Bethe cell
$\CY^{Bethe}(V;z)$.

\vsk.2>
Define the {\it totally positive Bethe subvariety} or {\it positive population}
as 
\bean
\label{tpBbb}
\CY^{Bethe}_{> 0}(V;z;\bb^{(0)}) := N_{>0} (\by^{(0)}))
\subset \mc Y^{Bethe}(V;z)\,.
\eean
We also have
\bean
\label{NB}
\CY^{Bethe}_{> 0}(V;z;\bb^{(0)}) = W^\dagger_{V,z}(\mc N_{>0}(V;z;\bb^{(0)})\,.
\eean

\subsection{Coordinates on the Bethe cell}

Let $\bb^{(0)}$ be a point of the Bethe cell  $\mc Y^{Bethe}(V;z)$.
Let $\bh = s_{i_q}\dots s_{i_1} \in \Red(w_0)$ be a reduced decomposition of the longest element $w_0\in S_{r+1}$\,.
\vsk.2>

We call the map 
\bea
\nu_\bh\,
:\ \BC^q \lra  \CY(V; z)^{Bethe},
\qquad  (c_q,\dots,c_1)\mapsto \nu_{i_q}(c_q)\dots\nu_{i_1}(c_1)\by^{(0)}\,,
\eea
the {\it Wronskian chart} corresponding to $\bh$. Its image 
is  Zariski open. The map $\nu_\bh$ is a birational isomorphism.

\vsk.2>
For any two reduced  words $\bh, \bh\rq{}\in\Red(w_0)$ we have
the transition function $\nu_{\bh\rq{}}^{-1}\circ \nu_\bh$, which defines an automorphism
\bean
\label{1,1,2n}
\mc R_{\bh, \bh'}:\ \F \iso  \F\,
\eean
of the field
$\F :=  \BC(c_1, \ldots, c_q)$\,.

\vsk.2>
Recall the Whitney-Lusztig charts on $N$, see \Ref{4.1.3},
\bea
\CL_\bh:\ \BC^q \lra N\,
\eea
and transition function automorphisms
\bea
R_{\bh, \bh'}:\ \F \iso  \F\,,
\eea
defined for any two words $\bh,\bh\rq{}\in\Red(w_0)$, see Section \ref{1.2}.

\begin{thm}
For any $\bh,\bh\rq{}\in\Red(w_0)$ we have
\bean
\mc R_{\bh, \bh'}\,=\, R_{\bh, \bh'}\,.
\eean
\end{thm}

\begin{proof}
The theorem follows from Comparison Theorem \ref{thm compTg}.
\end{proof}

\begin{cor}
\label{cor posp}
For any $\bh,\bh\rq{}\in\Red(w_0)$, the map 
 $\nu_{\bh\rq{}}\circ \nu_\bh^{-1}$ is well defined on
the positive population $\CY^{Bethe}_{> 0}(V;z;\bb^{(0)})$
and defines an isomorphism
\bea
\nu_{\bh\rq{}}\circ \nu_\bh^{-1} \,:\,
\CY^{Bethe}_{> 0}(V;z;\bb^{(0)})\,\to\, \CY^{Bethe}_{> 0}(V;z;\bb^{(0)})\,.
\eea

\end{cor}

\begin{proof}
The corollary follows from Theorem \ref{thm 1.1}
and Comparison Theorem \ref{thm compTg}.
\end{proof}

\section{Base affine space and  fat populations}
\label{7}

\subsection{Base affine space} 

Let $N, N_-\subset G := \on{SL}_{r+1}(\BC)$ be the subgroups of the upper and lower 
triangular matrices with $1$'s on the diagonal. 
Let $T\subset G$ be the subgroup of diagonal matrices.
Let $B_- = N_-T$ be the subgroup of lower triangular matrices and $B = NT $ the subgroup of upper triangular matrices

\vsk.2>
The quotient $G/N_-$   is called the {\it base affine space} of $G$. It is fibered over the flag space 
$G/B_-$ with fiber $T$. 

\vsk.2>
The image $\CB$ of $B$ in $G/N_-$ is called the {\it big cell} of the base affine space $G/N_-$\,.

\subsection{Fat population}
Let $V$ be an $r+1$-dimensional vector space as in Section \ref{sec poavs}.
\vsk.2>

The vector space $V$ has a {\it  volume form}. The {\it volume} of a basis $\bb=(b_1,\dots,b_{r+1})$ of $V$ is
defined to be the number
\bean
\label{volume}
\Wd (b_1,\dots,b_{r+1})\,.
\eean
Denote by  $B_V$ the set of all bases of $V$ of volume $1$. We consider every basis as an $r+1$-column vector.
Then the group $\on{SL}_{r+1}$ acts on $ B_V$ on the left freely with one orbit. The quotient
$B_V/N_{-}$ is isomorphic to the base affine space of $\on{SL}_{r+1}$.

\vsk.2>
The reduced Wronski map $W_V^\dagger$ defined in \Ref{Wd} induces a map
\bean
\label{wfat}
W_V^\dagger\,:\, B_V/N_-\,\to\, \C[x]^r\,.
\eean

\begin{thm}
\label{thm Wfat}
The reduced Wronski map $W_V^\dagger\,:\, B_V/N_-\,\to\, \C[x]^r\,$ is an embedding. The image of the map, denoted by 
$\on{Fat}(\mc Z_V)$, is isomorphic to the affine base space $G/N_-$  of the group $\on{SL}_{r+1}$.

\end{thm}

The image $\on{Fat}(\mc Z_V)$ will be called the {\it fat population} associated with $V$.

\begin{proof}
The proof is  parallel to the proof of Theorem \ref{thm 6.1} in \cite{MV}.
\end{proof}

We have another description of the fat population as a bundle 
\bean
\label{bundl}
\on{Fat}(\mc Z_V)\,\to\, \mc Z_V
\eean
over the population $\mc Z_V$, defined in Theorem
\ref{thm 6.1}, with fiber isomorphic to $(\C^\times)^r$. Namely, if
$\by= (y_1 :\dots : y_r) \in \mc Z_V \subset \Bbb P(\C[x])^r$ and
$(y_1 ,\dots , y_r) \in \C[x]^r$ is a representative of $\by$, then the fiber over $\by$
consist of the following $r$-tuples of polynomials
\bea
\{ (d_1y_1 ,\dots , d_ry_r) \in \C[x]^r\ | \ (d_1,\dots,d_r)\in (\C^\times)^r\}\,.
\eea

\subsection{Fat Bethe cell}

Let $\Si_V=\{z_1,\dots,z_n\}\subset \C$ be  the set of singular points of $V$.  Fix a  complex number
 $z\notin \Si_V$, a regular point of $V$. Let $\mc Y^{Bethe}(V;z) \subset \C[x]^r$ be the Bethe cell defined in 
\Ref{BcVz}. Define the {\it fat Bethe cell} $\on{Fat}(\mc Y^{Bethe}(V;z))$ by the formula
\bean
\label{FBC}
&&
\on{Fat}(\mc Y^{Bethe}(V;z)) =
\\
\notag
&&
\phantom{aaa}
=\{(d_1y_1 ,\dots , d_ry_r) \in \C[x]^r\ | \ 
(y_1 ,\dots , y_r)\in \mc Y^{Bethe}(V;z)\,,\ (d_1,\dots,d_r)\in (\C^\times)^r\}\,.
\eean

\vsk.2>
Recall  $\mc N(V;z)$,  the cell of bases of $V$ unipotent at $z$. Define the fat cell 
$\on{Fat}(\mc N(V;z))$ by the formula
\bean
\label{FUC}
&&
\on{Fat}(\mc N(V;z)) =\{(b_1d_1 , b_2d_2/d_1,\dots , b_rd_r/d_{r-1}, b_{r+1}/d_r) \in \C[x]^r\ |
\\
\notag
&&
\phantom{aaaaaaaaaaaaaaaaa}
 | \ 
(b_1 ,\dots , b_{r+1})\in \mc N(V;z)\,,\ (d_1,\dots,d_r)\in (\C^\times)^r\}\,.
\eean

\vsk.2>
Clearly the fat cell is isomorphic to the big cell $\mc B$ of the base affine space $G/N_-$.

\begin{thm}
The reduced Wronski map induces an
isomorphism
\bean
\label{FF}
W_V^\dagger\,:\, \on{Fat}(\mc N(V;z)) \,\to\,\on{Fat}(\mc Y^{Bethe}(V;z))\,.
\eean
Hence the fat Bathe cell $\on{Fat}(\mc Y^{Bethe}(V;z))$ is isomorphic to 
 the big cell $\mc B$ of the base affine space $G/N_-$.
\end{thm} 

\begin{proof}
This theorem is a corollary of Theorem \ref{thm NYB}.
\end{proof}

\subsection{Example of a cluster structure on a fat Bethe cell}
\label{7.4}  

Consider the example of the 3-dimensional vector space $V = \BC[x]_{\leq 2}$
of quadratic polynomials. In this case the set $\Si_V$ of singular points
of $V$ is empty.
We choose $z=0$, a regular point for $V$, and consider the corresponding
fat Bethe cell $\on{Fat}(\mc Y^{Bethe}(V;z))$. It consist of pairs of polynomials
\bean
\label{alb}
(\alpha_0 + \alpha_1x + \alpha_2x^2/2,  \beta_0 + \beta_1x + \beta_2x^2/2)
\eean
such that
\bea
\alpha_0 \ne 0,\qquad \beta_0\ne 0
\eea
and such that the Pl\"ucker equation holds,
$$
\alpha_1\beta_1 = \alpha_0\beta_2 + \alpha_2\beta_0\,.
$$
This is a familiar relation in the cluster algebra structure of type $A_1$ on the ring 
$\BC[\on{SL}_3/N_-]$, where the cluster 
variables are $\alpha_1, \beta_1$, cf. \cite[Section 3.1]{Z} and \cite{FZ}.

\vsk.2>

In fact, in this case the coefficients of the polynomials in
\Ref{alb} are nothing else 
but the Pl\"ucker coordinates on $\BC[\on{SL}_3/N_-]$.

\section{Appendix: Fourteen and Eightfold Ways}
\label{8}
\subsection{Group $\on{SL}_3$ and a 2-category} 

One can reformulate the Whitney-Lusztig data  from Section \ref{1.2} in the language of \cite{MS}. 

Namely, consider a {\it $2$-category}\footnote{for a definition of (globular) $n$-categories  see \cite{St}} 
$\CS_3$ whose objects are in bijection with $S_3$; the $1$-arrows correspond 
to the {\it weak Bruhat order} on this group: there are $6$ elementary arrows:
\bean
(123) \overset{\tau_{12}}\lra (213)  \overset{\tau_{23}}\lra (231)  
\overset{\tau'_{12}}\lra (321)
\label{8.1.1a}
\eean
and 
\bean
(123) \overset{\tau'_{23}}\lra (132)  \overset{\tau''_{12}}\lra (312)  
\overset{\tau''_{23}}\lra (321).
\label{8.1.1b}
\eean
Finally, there is one nontrivial $2$-arrow between two compositions
\bean
h: \tau'_{12}\tau_{23}\tau_{12} \lra \tau''_{23}\tau''_{12}\tau'_{23}\,.
\label{8.1.2}
\eean

This structure is conveniently visualized in a hexagon, to be denoted $\CP_3$: its $6$ vertices correspond to objects 
of $\CS_3$, $6$ oriented edges - to elementary $1$-morphisms $\tau$, 
and the unique $2$ cell - to the $2$-morphism $h$.

\vsk.2>

The Whitney-Lusztig  data described above may be called a {\it $2$-representation} $\rho$ of $\CS_3$: 
to each object we assign the same vector space $V_{(ijk)} = \Bbb F^3$ with a fixed standard base. To the elementary
arrows $\tau^{(\prime, \prime \prime)}_{ij}$ we assign  elementary matrices:
\bean
\rho(\tau_{12}) = e_1(a_1),\qquad \rho(\tau_{23}) = e_2(a_2),
\label{8.1.3}
\eean
etc.  To products of elementary arrows we assign the products of the corresponding matrices. 

Finally, $\rho(h)$ will be an automorphism of $\Bbb F$ given by 
\bean
\rho(h) = R_{121; 212} =: R\,,
\label{8.1.4}
\eean
see formulas \Ref{1.1.2a} and \Ref{1.1.2b}.
Note that 
\bean
R^2 = \Id_{\Bbb F},
\label{8.1.5}
\eean
see Section \ref{1.2}.

Thus, the matrix $\rho(\tau''_{23}\tau''_{12}\tau'_{23})$ is obtained from the matrix 
$\rho(\tau'_{12}\tau_{23}\tau_{12})$ by applying the field automorphism $\rho(h)$:
\bean
\rho(\tau''_{23}\tau''_{12}\tau'_{23}) = \rho(h)\bigl\{\rho(\tau'_{12}\tau_{23}\tau_{12})
\bigr\}.
\label{8.1.6}
\eean

Similarly, for any $r$ one defines in \cite{MS} an {\it $r$-category} $\CS_r$,
 whose $1$-coskeleton is a usual $1$-category corresponding to the symmetric group $S_{r+1}$ 
with the weak Bruhat order.  Informally speaking, $\CS_r$ is an $r$-category structure 
on the $r+1$-th {\it permutohedron} $\CP_{r+1}$, the $r$-dimensional polyhedron 
in $\BR^{r+1}$, the convex hull of a generic $S_{r+1}$-orbit of a point $\bx\in \BR^{r+1}$.  

In particular for each $w\in S_{r+1}$ we have an $r-1$-category $\CH om_{\CS_r}(e, w)$; its $1$-skeleton 
$$
\CH om^{\leq 1}_{\CS_r}(e, w) :=
\on{Sk}_1\CH om_{\CS_r}(e, w)
$$
is a usual ($1$-)category; the set of its objects is by definition the set 
$\overline{\Red(w)}$ which is obtained from $\Red(w)$ by identifying any two words $\bh$ and $\bh'$ if their 
only distinction  is a couple $ij$ in $\bh$ vs $ji$ in $\bh'$ somewhere in the middle, 
with $|i - j| > 1$. 

According to \cite{MS}, $\overline{\Red(w)}$  is equipped with a partial order, {\it the $2$nd Bruhat order} 
$\leq_2$. 
The category $\CH om^{\leq 1}_{\CS_r}(e, w) $ corresponds to this order, i.e. 
two words $\bh, \bh'\in \overline{\Red(w)}$ are connected by a unique arrow 
if and only if $\bh \leq_2 \bh'$.

\subsection{Group $\on{SL}_4$ and a tetrahedron equation}

\subsubsection{The Forteenfold Way}\footnote{Cf. \cite{GN}.}
\label{8.2}
 Consider the case $r = 3$. The $3$-category $\CS_4$ is related to a 
polyhedron $\CP_4$ which looks as follows. We take a regular octahedron and cut from it $6$ little square pyramids at its vertices; we get 
a polyhedron with $6\times 4 = 24$ vertices. It has $8$ hexagons and $6$ squares 
as its $2$-faces. 

The vertices of $\CP_4$ are in bijection with $S_4$. An  oriented edge $x \to y$ connects elements  
$x, y\in S_4$ such that $x\leq y$ and $\ell(y) = \ell(x) + 1$ where $\leq$ is the weak Bruhat order, and $\ell(x)$ is the usual 
length (the number of factors in a reduced decomposition). 

Its eight hexagonal $2$-faces correspond to identities  of the form
\bean
s_is_{i+1}s_i = s_{i+1}s_{i}s_{i+1}\,,
\label{8.2.1a}
\eean
which give rise to eight "elementary"\  $2$-morphisms of type \Ref{1.1.2a}.

Its six square edges correspond to identities of the form
\bean
s_is_j = s_js_i,\qquad |i - j| > 1.
\label{8.2.1b}
\eean
which give rise to the {\it identity} $2$-morphisms in $\mc S_4$. 

\vsk.2>
It is convenient to imagine the vertices $e$ and $w_0$ as the "North"\ and "South"\ poles 
of $\CP_4$. The set $\Red(w_0)$ consists of  $16$ longest paths on $\CP_4$, of length $6$, going downstairs, which connect $e$ and $w_0$. 

For example, the path $\ell_1$ is
$$
(1234) \overset{\tau_1}\lra (2134) \overset{\tau_2}\lra (2314) \overset{\tau'_1}\lra (3214) \overset{\tau_3}\lra (3241) \overset{\tau'_2}\lra (3421) \overset{\tau''_1}\lra (4321),  
$$
whereas the path $\ell_5$ is
$$
(1234) \overset{\tau'_3}\lra (1243) \overset{\tau''_2}\lra (1423) \overset{\tau''_3}\lra (1432) \overset{\tau'''_1}\lra (4132) \overset{\tau'''_2}\lra (4312) \overset{\tau'''_3}\lra (4321).  
$$

Fourteen  paths are depicted  below, it is a {\it $14$-fold Way}: 
\bean
\begin{matrix} 
 212321 & \lra & 213231 = 231231 =  231213 & \lra & 232123 \\ 
\uparrow & & & &  \downarrow\\
121321 & & & & 323123\\
|| & & & &  || \\
123121 & & & & 321323\\
\downarrow & & & & \uparrow \\
123212 & \lra & 132312 = 132132 =  312132 & \lra &  321232   
\end{matrix}\ \ .
\label{8.2.2}
\eean
The elementary paths which are equal as $1$-morphisms are connected by the signs $ = $.
They correspond to $6$ mutations of type \Ref{8.2.1b} and are in bijection with the square 
$2$-faces of $\CP_4$. If we identify the paths related by $=$, we are left with the set 
$$
Hom^{(1)}(e, w_0) = \overline{\Red(w_0)} \,,
$$
which contains $8$ elements. 

\vsk.2>
Let us number the elements of $\overline{\Red(w_0)}$ as $\ell_c,\ c\in \BZ/8\BZ$. 
These elements are connected by $8$ mutations of type 
\Ref{8.2.1a}, which are geometrically given by hexagons:
\bean
\begin{matrix} \ell_0 & \overset{h_1(1)}\lra & \ell_1 & \overset{h_2(3)}\lra & \ell_2 & \overset{h_1(3)}\lra & \ell_3 & \overset{h_2(1)}\lra & \ell_4\\
|| & & & & & & & & || \\
\ell_0 & \overset{h_1(4)}\lra & \ell_{-1} & \overset{h_2(2)}\lra & \ell_{-2} & \overset{h_1(2)}\lra & \ell_{-3} & \overset{h_2(4)}\lra & \ell_{-4}
\end{matrix}\ \  .
\label{8.2.3}
\eean
Thus, the $2$-nd Bruhat order on the set $\overline{\Red(w_0)}$  converts it to an octagon.

\vsk.2>
Finally we have one $3$-morphism (homotopy)
$$
k: \ h_2(1)h_1(3)h_2(3)h_1(1) \lra h_2(4)h_1(2)h_2(2)h_1(4)
$$
which corresponds to the single $3$-cell of $\CP_4$, its body. 

\subsubsection{Representation of the $3$-category}

\label{8.3}

The Whitney - Lusztig data give rise to a {\it $2$-representation} of the $3$-category 
$\CS_4$.  It is visualized on the permutohedron $\CP_4$ as follows. 

\vsk.2>
Our base field will be a purely transcendental extension of $\BC$,  
$$
\Bbb F = \BC(\ba) = \BC(a_1, a_2, a_3, a_4, a_5, a_6) \cong \BC(N_4).
$$
At each vertex we put the based $\Bbb F$-vector space $V =\Bbb F^6$. At the edges we put the elementary matrices 
$e_i(a_j) \in N_3(\Bbb F),\ 1\leq i \leq 3, 1\leq j \leq 6$. For example, at 
the  $3$ edges going down from $e$ we put 
the matrices $e_1(a_1), e_2(a_1), e_3(a_1)$, and to the last three edges coming to 
$w_0$ we put the matrices $e_1(a_6), e_2(a_6), e_3(a_6)$. 
To any path we assign the product of the corresponding matrices. 

For example
$$
\rho(\ell_1) = e_1(a_6)e_2(a_5)e_3(a_4)e_1(a_3)e_2(a_2)e_1(a_1),
$$
whereas
$$ 
\rho(\ell_2) = e_3(a_6)e_2(a_5)e_1(a_4)e_3(a_3)e_2(a_2)e_3(a_1),
$$

On $2$-faces of $\CP_4$ we put certain automorphisms of the base field $\Bbb F$. 

Namely, let us introduce involutive operators 
$L(i): \Bbb F\iso \Bbb F,\ 1\leq i \leq 5$, 
by 
\bean
L(i)(a_i)
& = &
a_{i+1},\qquad L(i)(a_{i+1}) = a_{i}, 
\label{8.3.1}
\\
\notag
L(i)(a_k) 
&=& a_{k}\qquad\ \ \ \text{ if}\ k \neq i, i+1, 
\eean
and $R(j):\ \Bbb F\iso \Bbb F, \  1\leq j\leq 4$, given by
\bean
\label{8.3.2}
R(j)(a_j) 
&=&
 a_{j+1}a_{j+2}/(a_j + a_{j+2}),
\\
\notag  R(j)(a_{j+1}) &=& a_{j} + a_{j+2},\ 
\\
\notag
R(j)(a_{j+2}) 
&=& a_{j}a_{j+1}/(a_j + a_{j+2}),
\\
\notag
R(j)(a_k) &=& a_{k}\quad \text{ if}\ k \neq i, i+1, 
\eean
cf. formula \Ref{1.1.2a}. 

On  eight hexagons (resp. on six squares) we put the involutions 
$R$ (resp. $L$) according to the picture:
\bean
\label{8.3.3}
&&
{}
\\
&&
\begin{matrix} 
 \Bbb F(212321) & \overset{R(3)}\lra & \Bbb F(213231) \overset{L(2)}\lra \Bbb F(231231) 
 \overset{L(5)}\lra  \Bbb F(231213) & 
 \overset{R(3)}\lra & \Bbb F(232123) \\ 
R(1)\uparrow & & & &  \downarrow R(1)\\
\Bbb F(121321) & & & & \Bbb F(323123)\\
L(3)\downarrow & & & & \downarrow L(3) \\
\Bbb F(123121) & & & & \Bbb F(321323)\\
R(4)\downarrow & & & & \uparrow R(4) \\
\Bbb F(123212) & \overset{R(2)}\lra & \Bbb F(132312) \overset{L(4)}\lra\Bbb F(132132) 
\overset{L(1)}\lra  \Bbb F(312132) & 
\overset{R(2)}\lra & \Bbb F(321232)   
\end{matrix} \ \ .
\notag
\eean
Here $\Bbb F(\bh), \bh\in \Red(w_0)$, is a copy of the field $\Bbb F$. 

\vsk.2>
This way we have associated to any path $\ell = \ell(\bh)$,  $\bh\in \Red(w_0)$, 
an upper triangular matrix $\rho(\ell)\in N(\Bbb F)$, and to every homotopy
$\ell\overset{h}\lra \ell'$
an automorphism 
$$
R(h):\ \Bbb F\iso \Bbb F
$$
such that    
\bean
\rho(\ell') = R(h)\{\rho(\ell)\}
\label{8.3.4}
\eean

\smallskip
\begin{thm}
\label{8.4}
The diagram \Ref{8.3.3} is commutative, i.e.,
$$
L(3)R(1)R(3)L(2)L(5)R(3)R(1) = R(4)R(2)L(1)L(4)R(2)R(4)L(3)\,.
$$
\end{thm}

\smallskip

\subsubsection{Eightfold Way}
\label{8.5}

 We can rewrite this assertion as follows.

Define $8$ automorphisms (involutions or compositions of two involutions):
\bea
&&
R_1(1) = R(1),\qquad\ \ \ \  \ \ R_2(1) = L(3)R(1),\ 
\\
&&
R_1(2) = L(4)R(2),\qquad 
R_2(2) = R(2)L(1),
\\
&&
R_1(3) = L(5)R(3),\qquad R_2(3) = R(3)\bL(2),\ 
\\
&&
R_1(4) = R(4)L(3),\qquad R_2(4)= R(4).
\eea
Then we have a {\it tetrahedron equation}
\bean
R_2(1)R_2(3)R_1(3)R_1(1) = R_2(4)R_2(2)R_1(2)R_1(4)\,.
\label{8.5.1}
\eean

\end{document}